\documentclass[12pt]{article}

\usepackage{algorithm,algorithmicx,algpseudocode}

\usepackage{graphicx}
\usepackage{changebar} 
\usepackage{multirow}
\usepackage{multicol}
\usepackage{latexsym,amsmath,amsfonts,amscd, amsthm,esint}
\usepackage{changebar}
\usepackage{subfigure}
\usepackage{color}
\usepackage{tikz}
\usetikzlibrary{arrows,backgrounds,snakes,shapes}

\newtheoremstyle{plainNoItalics}{}{}{\normalfont}{}{\bfseries}{.}{ }{}

\theoremstyle{plain}
\newtheorem{thm}{Theorem}[section]
\theoremstyle{plainNoItalics}

\newtheorem{defn}[thm]{Definition} 
\newtheorem{rem}[thm]{Remark}
\newtheorem{prop}[thm]{Proposition}
\newtheorem{exa}[thm]{Example}

\newcommand{\bx}{{\bf x}}
\newcommand{\f}{\frac}

\newcommand{\beq}{\begin{equation}}
\newcommand{\eeq}{\end{equation}}
\newcommand{\beqa}{\begin{eqnarray}}
\newcommand{\eeqa}{\end{eqnarray}}
\newcommand{\bit}{\begin{itemize}}
\newcommand{\eit}{\end{itemize}}
\newcommand{\bedef}{\begin{defn}}
\newcommand{\edefn}{\end{defn}}
\newcommand{\bpro}{\begin{prop}}
\newcommand{\epro}{\end{prop}}
\newcommand{\Dx}{\Delta x}
\newcommand{\Dy}{\Delta y}
\newcommand{\Dt}{\Delta t}

\newcommand{\xL}{{x_{j-\frac{1}{2}}}}
\newcommand{\xR}{{x_{j+\frac{1}{2}}}}

\newcommand{\jL}{{j-\frac{1}{2}}}
\newcommand{\jR}{{j+\frac{1}{2}}}

\tikzstyle{myarrow}=[->, >=open triangle 90, thick]
\tikzstyle{compoarrow}=[->,>=diamond,thick ]
\tikzstyle{inhearrow}=[->,>=diamond,thick ]
\tikzstyle{line}=[-, thick]
\tikzstyle{arrowline} = [draw, -latex']
\tikzstyle{arrowline1} = [draw, latex'-]
\tikzstyle{decision} = [diamond, aspect = 3, text centered, draw=black]
\tikzstyle{abstract}=[rectangle, draw=black,text centered, anchor=north, minimum width = 3cm]
\tikzstyle{abstract1}=[rectangle, draw=black,anchor=north, text width=4cm]
\tikzstyle{comment}=[rectangle, draw=black, rounded corners, fill=green, drop shadow, text centered, anchor=north, text=white, text width=3cm]
\tikzstyle{myarrow}=[->, >=open triangle 90, thick]
\tikzstyle{compoarrow}=[->,>=diamond,thick ]
\tikzstyle{inhearrow}=[->,>=diamond,thick ]
\tikzstyle{line}=[-, thick]
\tikzstyle{arrow} = [->,>=stealth]

\begin{document}

\baselineskip=1.8pc

%=============  title  =========================

\begin{center}
{\bf
A semi-Lagrangian discontinuous Galerkin (DG) - local DG method for solving convection-diffusion equations
}
\end{center}

\vspace{.03in}

\vspace{.2in}
\centerline{
Mingchang Ding\footnote{
 Department of Mathematical Sciences, University of Delaware, Newark, DE, 19716. E-mail: dmcvamos@udel.edu.
},
Xiaofeng Cai\footnote{
 Department of Mathematical Sciences, University of Delaware, Newark, DE, 19716. E-mail: xfcai@udel.edu.
},
 Wei Guo\footnote{
Department of Mathematics and Statistics, Texas Tech University, Lubbock, TX, 70409. E-mail:
weimath.guo@ttu.edu. Research
is supported by NSF grant NSF-DMS-1830838 (Program manager: Dr.~Leland M. Jameson).
},
Jing-Mei Qiu\footnote{Department of Mathematical Sciences, University of Delaware, Newark, DE, 19716. E-mail: jingqiu@udel.edu. Research of the first, second and last author is supported by NSF grant NSF-DMS-1522777 and NSF-DMS-1818924 (Program manager: Dr.~Leland M. Jameson), Air Force Office of Scientific Research FA9550-18-1-0257 (Program manager: Dr.~Fariba Fahroo).}
}
\vspace{.05in}

\centerline{\bf Abstract}

In this paper, we propose an efficient high order semi-Lagrangian (SL) discontinuous Galerkin (DG) method for solving linear convection-diffusion equations. The method generalizes our previous work on developing the SLDG method for transport equations  \cite{cai2017high}, making it capable of handling additional diffusion and source terms. Within the DG framework, the solution is evolved along the characteristics; while the diffusion term is discretized by the local DG (LDG) method and  integrated along characteristics by implicit Runge-Kutta methods together with source terms. The proposed method is named the `SLDG-LDG' method and enjoys many attractive features of the DG and SL methods. These include the uniformly high order accuracy (e.g. third order) in space and in time, compact, mass conservative, and stability under large time stepping size.  An $L^2$ stability analysis is provided when the method is coupled with the first order backward Euler discretization. Effectiveness of the method are demonstrated by a group of numerical tests in one and two dimensions.

\vspace{0.2in}


\newpage

\section{Introduction}
\label{sec1}
\setcounter{equation}{0}

In this paper, we are concerned with
 solving the time dependent convection-diffusion problems in the form of
\beq \label{eq3:1}
\begin{cases}
	u_t +\nabla_{\bx} \cdot ({\bf a}(\bx,t) u) 			= 	\epsilon \Delta u + g,	\quad \bx \in \Omega,\quad t>0, \\
	u(\bx,0) 									= 	u_0(\bx),\quad \bx \in \Omega
\end{cases}
\eeq
with $\epsilon \geq 0$. For the scope of our current research, we assume the velocity field ${\bf a}(\bx, t)$ to be continuous with respect to $\bx$ and $t$.

A popular computational method for finding approximate solutions to transport dominant problems in the form \eqref{eq3:1} is the semi-Lagrangian (SL) method, which has a long history in computational fluid dynamics, e.g. for convection-diffusion problems \cite{xiu2001semi, russell2002overview}, climate modeling \cite{lin1996multidimensional, staniforth1991semi, diamantakis2013semi}, plasma simulations \cite{sonnendrucker1999semi}, as well as linear and Hamilton-Jacobi equations \cite{falcone2013semi}. For transport dominant problems, the method is designed via tracking the characteristics forward or backward in time, thus avoiding the time step restriction,  and can be coupled with various  spatial discretization, such as the finite element method \cite{russell2002overview}, the finite difference method with polynomial and spline interpolations \cite{sonnendrucker1999semi}, the spectral element method \cite{giraldo2003spectral}, the discontinuous Galerkin (DG) method \cite{cai2017high}. In the presence of diffusion and source terms, usually time integration should be performed along characteristics, e.g. see \cite{groppi2016high} for the BGK model, and \cite{celia1990eulerian, bonaventura2016flux, xiu2001semi} for linear and nonlinear convection-diffusion models. 

The objective of this paper is to develop an efficient high order SL method for \eqref{eq3:1} under the DG framework. The DG discretization approach is a class of finite element methods that use piecewise continuous approximations and enjoy many attractive computational advantages for transport dominant problems. 
In this paper, we propose to evolve the convection term by the SLDG method recently proposed in \cite{cai2017high}, and treat the diffusion term by the local DG (LDG) method coupled with a diagonally implicit (DI) Runge-Kutta (RK) method along dynamic characteristics elements. 
The proposed method is termed as the SLDG-LDG method. In the scheme formulation, we introduce the adjoint problem for the test function in the same spirit of ELLAM \cite{russell2002overview}, and project the DG solution and LDG approximation to second derivative terms onto a set of time-dependent characteristics elements, based on the procedure developed in our earlier work \cite{cai2017high}. There are a few key novelties of this work, compared with existing methods in the literature. First, thanks to the DG  framework together with the backward characteristics tracing mechanism, our proposed scheme is inherently locally mass conservative. In particular, when compared with the SL finite difference framework in which a high order interpolation is employed, the DG finite element and the finite volume schemes are known to be a more natural framework for mass conservation; the authors in \cite{bosler2019conservative} propose an SLDG method with forward characteristics tracing. In their work the global mass conservation is enforced by an extra step of constrained optimization, i.e., a mass fixer, and hence the local mass conservation as well as the original order accuracy are not guaranteed. We note that the SLDG work \cite{cai2017high} is an extension of the CSLAM \cite{lauritzen2010conservative} from the finite volume setting to the DG setting by introducing an adjoint problem for the test function; and it shares the same local mass conservation property with the CSLAM. 
Second, we inherit advantages of the DG in the SL framework. These include the schemes' ability to resolve solution structures and to evaluate the diffusion term by the LDG method. In the LDG method, by introducing auxiliary variables the second spatial derivative is rewritten into a system of first order equations, and proper choices of fluxes are made for numerical stability and accuracy. Third, unlike the method-of-lines approach in an Eulerian framework, the time integration has to be performed for the material derivative which is not necessarily aligned with the background grid. Hence, extra effort has to be made. Second order Crank-Nicolson and BDF methods have been proposed and used in the SL setting \cite{xiu2001semi}; yet there is little existing work that employs higher than second order multi-stage RK method for the integration of non-convection terms. In this work, on each RK stage, we propose to update the solution and the diffusion term on the background elements; and then project them to the characteristics elements by the same SLDG algorithm in a purely convective setting \cite{cai2017high}. The RK implementation can be done in a stage-by-stage manner as that of first order backward Euler method. As we use the SL method for transport and implicit RK method along characteristics for other terms, our scheme is highly accurate and unconditionally stable for linear problems. Last, our scheme formulation does not employ operator splitting and thus is free of splitting error.

Another class of very popular solvers for \eqref{eq3:1} is the Eulerian method, among which the most relevant high order methods related to this work is the Eulerian DG method. 
Typically, an implicit-explicit (IMEX) RK time discretization is used for time discretization of \eqref{eq3:1}, i.e. the convection term is handled by explicit  RK methods, while the diffusion term is discretized by an LDG \cite{cockburn1998local} method in space along with an implicit RK method in time. From the stability analysis via the energy method in \cite{wang2015stability}, there is a very strong result stating that ``such IMEX LDG schemes are unconditionally stable for the linear problems in the sense that the time-step size is only required to be upper-bounded by a constant which depends on the ratio of the diffusion and the square of the advection coefficients and is independent of the spatial mesh-size h, even though the advection term is treated explicitly." We remark that, under the same setting, our scheme is unconditionally stable with no time step constraint for stability, when a first order backward Euler method is used. Extension of the theoretical analysis, when a higher order DIRK method is used for diffusion term, is subject to future work.

The rest of the paper is organized as follows. In Section 2, we introduce the proposed methodology for one-dimensional (1D) and two-dimensional (2D) problems; theoretically we prove the mass conservation and $L^2$ stability when the method is coupled with the first-order backward Euler method. In Section 3, we present numerical results to demonstrate the effectiveness of our proposed approach with high order accuracy, and stability under large time stepping sizes. Finally, a conclusion is given in Section 4.

\section{The SLDG-LDG method for convection-diffusion problems}
\label{sec2}

\setcounter{equation}{0}

In this paper, we focus on problems in one and two dimensions on rectangular domains with zero or periodic boundary conditions. Notice that our problem \eqref{eq3:1} is in the conservative form, for which local mass conservation is desired at the discrete level for the numerical scheme. 
Below, we formulate the proposed scheme for 1D problems in Section~\ref{1D_form} by first introducing the spatial discretization and the adjoint problem for the test function; then we introduce the proposed treatment of the diffusion and source terms with DIRK methods along characteristic elements. The extensions to 2D problems are then discussed briefly in Section~\ref{2D_form}.

%------------------------------------------------------------------------------------------------------------------------------------------------------------------------------------------------------------
\subsection{Scheme formulation: 1D case}
\label{1D_form}

To introduce the algorithm, we start from the 1D case of \eqref{eq3:1}:
\beq \label{eq3:1:1d}
	u_t +(a(x,t) u)_x 	=	 \epsilon u_{xx} + g.
\eeq

\noindent
{\bf I. Spatial discretization: DG solution and test function spaces.}
We discretize the 1D domain $[x_a, x_b]$ into $N$ elements:
$x_a 	= x_{\frac{1}{2}} 	<x_{\frac{3}{2}} <	 \cdots < 	x_{N+\frac{1}{2}} = 	x_b,$
with $I_j = [\xL,\xR]$ denoting an element of length $\bigtriangleup x_{j}=\xR-\xL$ for $j=1, 2, \cdots, N$. $\Dt = t^{n+1}-t^n$ represents the time discretization step. In the framework of the DG method, we let numerical solutions and test functions belong to the finite dimensional piecewise approximation space
\beq \label{eq2:3}
	V^k_h =  \{v_h:v_h|_{I_j} 	\in P^k(I_j),\, j =1, 2, \cdots, N \},
\eeq
where $P^k(I_j)$ denotes the set of polynomials of degree at most $k$ over $I_j$.

\noindent
{\bf II. Adjoint problem.}
To formulate the SLDG-LDG scheme, we follow a similar idea in \cite{guo2014conservative, cai2017high} by considering the following adjoint problem for the test function $\psi(x, t)$ that
satisfies
\beq \label{eq3:3}
	\psi_t+ a(x,t) \psi_x	=	0, \quad t \in [\tau_1, \tau_2]
\eeq
with
\beq
\label{eq3:3a}
\psi(x, \tau_2) = \Psi(x) \in V^k_h.
\eeq
In other words, $\psi$ satisfies a final-value problem with function values specified at $\tau_2$. For a pure convection problem \cite{cai2017high}, we have $[\tau_1, \tau_2] = [t^n, t^{n+1}]$; while $\tau_1$ and $\tau_2$ could also correspond to different time stages in an implicit RK method when discretizing the diffusion and source terms along characteristics.
Next, we make the following observations for the test function $\psi(x,t)$:
\bit
\item[(i)] While the original problem \eqref{eq3:1} is in the conservative form, an adjoint problem for the test function is in the
advective form \eqref{eq3:3}. Along characteristics curves governed by
\beq
\label{eq: 3:4}
\frac{d \widetilde{x}(t)}{dt} = a(\widetilde{x}(t), t),
\eeq
$\psi(\widetilde{x}(t),t)$ stays constant. Hence $\psi(x,t)$, $\forall x\in[x_a, x_b]$, $t \in [\tau_1, \tau_2)$ can be obtained by tracking characteristics based on \eqref{eq: 3:4}. 
\item[(ii)] The test function satisfies a final value problem \eqref{eq3:3a}. In general, $\psi(x,t)$ with $t \in [\tau_1, \tau_2)$, is not necessarily a polynomial. Yet, it can be approximated by polynomials with high order accuracy as presented in the algorithm flowchart Step 1.1 below.
\eit

\noindent
{\bf III. Time dependent characteristics interval, see Figure~\ref{schematic_1d(a)}.}
Let
$$\widetilde{I}_j^{n+1,n}(t) =[\widetilde{x}_{\jL}(t), \widetilde{x}_{\jR}(t)], \quad t\in [t^n, t^{n+1}],$$
 be the dynamic interval bounded by characteristics curves $\widetilde{x}_{\jL}(t)$ and $\widetilde{x}_{\jR}(t)$ emanating from cell boundaries of $I_j$ at $t^{n+1}$, where $\widetilde{x}_{j \pm \f12}(t)$ satisfy the final value problems
\beq
\label{eq: char_inter}
	\frac{d \widetilde{x}_{j \pm \f12} (t)}{dt} = a(\widetilde{x}_{j \pm \f12} (t), t), \qquad
	\widetilde{x}_{j \pm \f12} (t^{n+1}) = x_{j\pm\f12}.
\eeq
Let $[\widetilde{x}_{\jL}(t^n), \widetilde{x}_{\jR}(t^n)] = [x_{j-\frac12}^{n+1,n},x_{j+\frac12}^{n+1,n}]$, which will be referred to as the ``upstream cell" at $t^n$ later.
Here the superscripts {\small{$n+1, n$}} refer to the interval from $t^{n+1}$, being tracked backward in time to $t^n$.
See Figure~\ref{schematic_1d(a)} for illustration of $\widetilde{x}_{\jL}(t)$, $\widetilde{x}_{\jR}(t)$ and $[x_{j-\frac12}^{n+1,n},x_{j+\frac12}^{n+1,n}]$.

\begin{figure}[h!]
\centering
\subfigure[]{
\label{schematic_1d(a)}
\begin{tikzpicture}[scale=1.0]
%%%%%%%%%%%%%%%%%%%%%%%%%%%%%%%%%%%%%%%%%%%%%%%%%%%%%%%%%%%%%%%%%%%%%%%%%%%%%%%%%%%%%
% draw the integration region as a background
    \draw[white,fill=blue!3] (0,3) to[out=210,in=80] (0-2.,0)  -- (2.5-2,0) to[out=80,in=210] (2.5 ,3)
      -- cycle;
%%%%%%%%%%%%%%%%%%%%%%%%%%%%%%%%%%%%%%%%%%%%%%%%%%%%%%%%%%%%%%%%%%%%%%%%%%%%%%%%%%%%%
% draw two stage, t^n and t^{n+1}
    \draw[black]                 (-3,0) node[left] { } -- (3,0)
                                        node[right]{\scriptsize$t^{n}$};
    \draw[black]                 (-3,3) node[left] { } -- (3,3)
                                        node[right]{\scriptsize$t^{n+1}$};
%%%%%%%%%%%%%%%%%%%%%%%%%%%%%%%%%%%%%%%%%%%%%%%%%%%%%%%%%%%%%%%%%%%%%%%%%%%%%%%%%%%%%%
% draw ticks
% 3 evenly spaced ticks at t^{n+1} using a length 2.5
    \draw[thick] (0, 3-0.1) -- (0, 3+0.1) node[above] {\scriptsize$x_{j-\frac12}$};
    \draw[thick] (2.5, 3-0.1) -- (2.5, 3+0.1) node[above] {\scriptsize$x_{j+\frac12}$};
    \draw[thick] (-2.5, 3-0.1) -- (-2.5, 3+0.1) node[above] {};
% 3 evenly spaced ticks at t^{n+1} using a length 2.5
    \draw[thick] (0, 0-0.1) -- (0, 0+0.1) node[above] { };
    \draw[thick] (2.5, 0-0.1) -- (2.5, 0+0.1) node[above] { };
    \draw[thick] (-2.5, 0-0.1) -- (-2.5, 0+0.1) node[above] { };
%%%%%%%%%%%%%%%%%%%%%%%%%%%%%%%%%%%%%%%%%%%%%%%%%%%%%%%%%%%%%%%%%%%%%%%%%%%%%%%%%%%%%%
% draw two characteristic lines
 \draw[-latex,dashed,blue]( 0, 3 )node[left,scale=1.3]{$$}
        to[out=210,in=80] ( 0-2,0) node[below=2pt] { };
 \draw[-latex,dashed,blue]( 2.5, 3 )node[left,scale=1.3]{$$}
        to[out=210,in=80] ( 2.5-2,0) node[below=2pt] { };
%%%%%%%%%%%%%%%%%%%%%%%%%%%%%%%%%%%%%%%%%%%%%%%%%%%%%%%%%%%%%%%%%%%%%%%%%%%%%%%%%%%%%%
% draw nodes
 \fill [blue] ( 0, 3 ) circle (1.6pt) node[right] {};
 \fill [blue] ( 2.5, 3 ) circle (1.6pt) node[right] {};

 \fill [red] ( 0-2, 0 ) circle (1.6pt) node[above left=-0.1 ] {\scriptsize$x_{j-\frac12}^{n+1,n}$};
 \fill [red] ( 2.5-2, 0 ) circle (1.6pt) node[above right=-0.1] {\scriptsize$x_{j+\frac12}^{n+1,n}$};
%%%%%%%%%%%%%%%%%%%%%%%%%%%%%%%%%%%%%%%%%%%%%%%%%%%%%%%%%%%%%%%%%%%%%%%%%%%%%%%%%%%%%%%
% denote subintervals
\draw [decorate,color=red,decoration={brace,mirror,amplitude=3pt},xshift=0pt,yshift=0pt]
(-2. ,0) -- (0,0) node [red,midway,xshift=0cm,yshift=-14pt]
{\scriptsize $I_{j,1}^{n+1,n}$};
\draw [decorate,color=red,decoration={brace,mirror,amplitude=3pt},xshift=0pt,yshift=0pt]
( 0 ,0) -- (0.5,0) node [red,midway,xshift=0cm,yshift=-14pt]
{\scriptsize$I_{j,2}^{n+1,n}$};
%%%%%%%%%%%%%%%%%%%%%%%%%%%%%%%%%%%%%%%%%%%%%%%%%%%%%%%%%%%%%%%%%%%%%%%%%%%%%%%%%%%%%%%
\node[blue!70, rotate=0] (a) at ( 0.2 ,2.2) { \large $ \mathcal{K}$ };
%%%%%%%%%%%%%%%%%%%%%%%%%%%%%%%%%%%%%%%%%%%%%%%%%%%%%%%%%%%%%%%%%%%%%%%%%%%%%%%%%%%%%%%
\draw[-latex,blue!80]( -1.2,1)node[right=-2pt,scale=1.]{\scriptsize $\widetilde{I}_j^{n+1,n}(t)$}
        to[out=180,in=0] (-1.75,1) node[above left=2pt] {$$};  

\draw[-latex,blue!80](0.2,1)node[right,scale=1.]{ }
        to[out=0,in=180] (0.75,1) node[above left=2pt] {$$};

\end{tikzpicture}
}
\subfigure[]{
\label{schematic_1d(b)}
\begin{tikzpicture}[scale=1.0]
%%%%%%%%%%%%%%%%%%%%%%%%%%%%%%%%%%%%%%%%%%%%%%%%%%%%%%%%%%%%%%%%%%%%%%%%%%%%%%%%%%%%%
% draw the integration region as a background
    \draw[white,fill=blue!3] (0,3) to[out=210,in=80] (0-2.,0)  -- (2.5-2,0) to[out=80,in=210] (2.5 ,3)
      -- cycle;
%%%%%%%%%%%%%%%%%%%%%%%%%%%%%%%%%%%%%%%%%%%%%%%%%%%%%%%%%%%%%%%%%%%%%%%%%%%%%%%%%%%%%
% draw two stage, t^n and t^{n+1}
    \draw[black]                 (-3,0) node[left] { } -- (3,0)
                                        node[right]{\scriptsize$t^{n}$};
    \draw[black]                 (-3,3) node[left] { } -- (3,3)
                                        node[right]{\scriptsize$t^{n+1}$};
%%%%%%%%%%%%%%%%%%%%%%%%%%%%%%%%%%%%%%%%%%%%%%%%%%%%%%%%%%%%%%%%%%%%%%%%%%%%%%%%%%%%%%
% draw ticks
% 3 evenly spaced ticks at t^{n+1} using a length 2.5
    \draw[thick] (0, 3-0.1) -- (0, 3+0.1) node[above] {\scriptsize$x_{j-\frac12}$};
    \draw[thick] (2.5, 3-0.1) -- (2.5, 3+0.1) node[above] {\scriptsize$x_{j+\frac12}$};
    \draw[thick] (-2.5, 3-0.1) -- (-2.5, 3+0.1) node[above] {};
% 3 evenly spaced ticks at t^{n+1} using a length 2.5
    \draw[thick] (0, 0-0.1) -- (0, 0+0.1) node[above] { };
    \draw[thick] (2.5, 0-0.1) -- (2.5, 0+0.1) node[above] { };
    \draw[thick] (-2.5, 0-0.1) -- (-2.5, 0+0.1) node[above] { };
%%%%%%%%%%%%%%%%%%%%%%%%%%%%%%%%%%%%%%%%%%%%%%%%%%%%%%%%%%%%%%%%%%%%%%%%%%%%%%%%%%%%%%
% draw two characteristic lines
 \draw[-latex,dashed,blue]( 0, 3 )node[left,scale=1.3]{$$}
        to[out=210,in=80] ( 0-2,0) node[below=2pt] { };
 \draw[-latex,dashed,blue]( 2.5, 3 )node[left,scale=1.3]{$$}
        to[out=210,in=80] ( 2.5-2,0) node[below=2pt] { };
%%%%%%%%%%%%%%%%%%%%%%%%%%%%%%%%%%%%%%%%%%%%%%%%%%%%%%%%%%%%%%%%%%%%%%%%%%%%%%%%%%%%%%
% draw nodes
 \fill [blue] ( 0, 3 ) circle (1.6pt) node[right] {};
 \fill [blue] ( 2.5, 3 ) circle (1.6pt) node[right] {};

 \fill [red] ( 0-2, 0 ) circle (1.6pt) node[above left=-0.1 ] {\scriptsize$x_{j-\frac12}^{n+1,n}$};
 \fill [red] ( 2.5-2, 0 ) circle (1.6pt) node[above right=-0.1] {\scriptsize$x_{j+\frac12}^{n+1,n}$};
%%%%%%%%%%%%%%%%%%%%%%%%%%%%%%%%%%%%%%%%%%%%%%%%%%%%%%%%%%%%%%%%%%%%%%%%%%%%%%%%%%%%%%%
% draw
 \draw[-latex ,blue]( 0+0.6, 3 )node[left,scale=1.3]{$$}
        to[out=210,in=80] ( 0-2+0.6,0) node[below=2pt] { };
 \draw[-latex, blue]( 2.5-0.6, 3 )node[left,scale=1.3]{$$}
        to[out=210,in=80] ( 2.5-2-0.6,0) node[below=2pt] { };
%%%%%%%%%%%%%%%%%%%%%%%%%%%%%%%%%%%%%%%%%%%%%%%%%%%%%%%%%%%%%%%%%%%%%%%%%%%%%%%%%%%%%%%
% draw nodes
 \fill [blue] ( 0+0.6, 3 ) circle (1.6pt) node[right] {};
 \fill [blue] ( 2.5-0.6, 3 ) circle (1.6pt) node[right] {};

 \fill [red] ( 0-2+0.6, 0 ) circle (1.6pt) node[above left=-0.1 ] { };
 \fill [red] ( 2.5-2-0.6, 0 ) circle (1.6pt) node[above right=-0.1] { };
 %%%%%%%%%%%%%%%%%%%%%%%%%%%%%%%%%%%%%%%%%%%%%%%%%%%%%%%%%%%%%%%%%%%%%%%%%%%%%%%%%%%%%%%
 \draw[-latex ]( 0+1, 3.5 )node[right=-3pt,scale=1.0]{\scriptsize$x_{j,i_{gl}}$}
        to[out=210,in=80] ( 0+0.6, 3) node[below=2pt] { };
%%%%%%%%%%%%%%%%%%%%%%%%%%%%%%%%%%%%%%%%%%%%%%%%%%%%%%%%%%%%%%%%%%%%%%%%%%%%%%%%%%%%%%
 \draw[-latex,red]( -1.5 , -0.5 )node[right=-4pt,scale=1.0]{\scriptsize$(x_{j,i_{gl}}^{n+1,n}, \Psi(x_{j,i_{gl} })) \rightarrow \Psi_\star^{n+1,n}(x) \text{ interpolate} $ }
        to[out=140,in=220] ( 0+0.6-2, 0) node[below=2pt] { };

\end{tikzpicture}
}

\caption{Schematic illustration of the SLDG-LDG formulation in 1D. Left: Integration region $\mathcal{K}$, dynamic interval $\widetilde{I}_j^{n+1,n}(t)$ and upstream interval $I^{n+1,n}_j = I^{n+1,n}_{j,1} \cup I^{n+1,n}_{j,2}$. Right: Interpolation of $\psi^{n+1,n}$.}

\label{schematic_1d}
\end{figure}
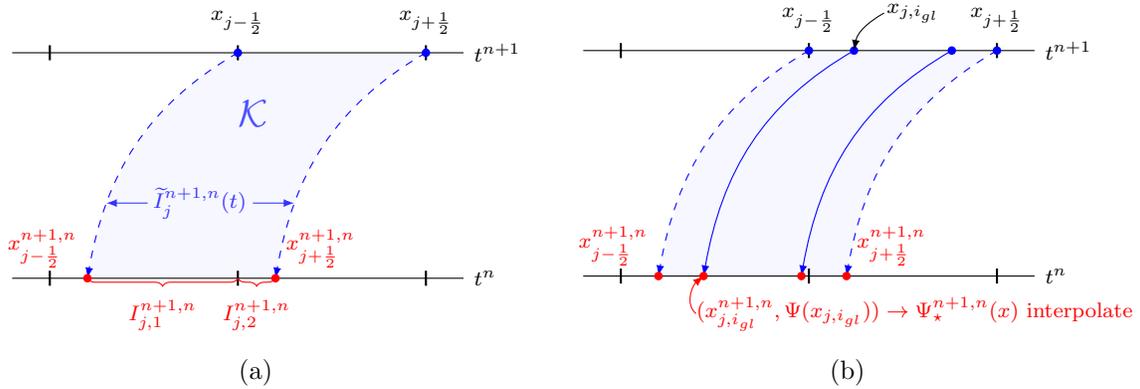

\noindent
{\bf IV. SLDG-LDG scheme formulation and discretization.}
Let $u^{n}$ be the numerical solution at time $t^{n}$, the SLDG-LDG scheme is formulated based on Proposition \ref{prop_1D} presented below.

\begin{prop} \label{prop_1D}

Consider the 1D problem \eqref{eq3:1:1d} and the adjoint problem \eqref{eq3:3} for the test function $\psi$,
then the following identity holds
\beq \label{eq3:3:2}
\int_{I_j}u^{n+1}\Psi \ dx	-	\underbrace{\int_{I^{n+1,n}_j } u^n\psi^{n+1, n} \ dx}_{\text{Term I}}	=	\underbrace{\int^{t^{n+1}}_{t^n}\int_{\widetilde{I}^{n+1,n}_j(t)} [\epsilon u_{xx}+g] \psi \ dx\ dt}_{\text{Term II}}, \quad \Psi \in P^k(I_j)
\eeq
where $\psi^{n+1, n}$ is the solution to the adjoint problem \eqref{eq3:3} with $\psi(x, t^{n+1}) = \Psi$ at $t^n$.

\end{prop}

\begin{proof}

$\int_{ \mathcal{K} } \eqref{eq3:1:1d} \cdot \psi		+	\eqref{eq3:3} \cdot u \ dx\ dt$ leads to
\beq \label{eq3:3:3}
	\int_{\mathcal{K}} \left([ u_t+(a(x,t)u)_x ] \psi+[ \psi_t+a(x,t)\psi_x]u\right) \ dx\ dt = \int_{\mathcal{K} } [\epsilon u_{xx}+g] \psi \ dx\ dt,
\eeq
where the domain $\mathcal{K} $ (see Figure $\ref{schematic_1d(a)}$) is bounded by $I_j$, $I^{n+1,n}_j $ from above and below, and characteristic trajectories $\widetilde{x}_{\jL}(t)$ and $\widetilde{x}_{\jR}(t)$ from left and right, respectively. Rearranging terms in eq. \eqref{eq3:3:3} gives
\beq \label{eq3:3:1}
	\int_{\mathcal{K}}[u \psi]_t+[a(x,t)u \psi]_x \ dx\ dt = \int_{\mathcal{K} } [\epsilon u_{xx}+g] \psi \ dx\ dt.
\eeq
Applying the divergence theorem to the left-hand side (LHS) of eq. \eqref{eq3:3:1} and due to the cancellation of the integrals along the characteristic curves, we prove \eqref{eq3:3:2}.

\end{proof}
To update $u^{n+1} \in V^k_h$, one has to evaluate Term I and Term II in eq. \eqref{eq3:3:2} by letting the test function $\Psi$ go through all the basis functions in $V^k_h$.  In particular, the proposed SLDG-LDG scheme consists of the following two steps: Step 1.1 and 1.2.

\noindent
{\bf Step 1.1: Evaluation of Term I of \eqref{eq3:3:2} as in the SLDG method \cite{cai2017high}.}
To evaluate $\int_{I^{n+1,n}_j} u^n\psi^{n+1,n} \ dx$, we propose the procedures below.

	\noindent
	{\bf Step 1.1a: Reconstruct test function $\psi^{n+1,n}$ through interpolation.}		
	Choose $k+1$ interpolation points  $\{x_{j,i_{gl}}\}^{k+1}_{i_{gl} = 1}$ such as the Gauss-Lobatto (GL) points over $I_j$ at $t^{n+1}$ and locate the characteristic feet $\{x^{n+1,n}_{j,i_{gl} }\}^{k+1}_{i_{gl} = 1}$ at $t^n$ by solving
	\beq
		\begin{cases}
		\frac{d \widetilde{x} (t)}{dt}	=	a(\widetilde{x} (t),t)\\
		\widetilde{x} (t^{n+1})		=	x_{j,i_{gl} }
		\end{cases}
	\eeq
	with high order numerical integrators. In our implementation, a fourth order RK method is applied. 
	Thus, $\psi^{n+1,n}(x^{n+1,n}_{j, i_{gl} }) = \Psi(x_{j,i_{gl} }), \text{for\ } i_{gl} = 1, \cdots, k+1$. 
	We then construct a degree $k$ polynomial $\Psi^{n+1,n}_\star(x)$ interpolating the $\psi^{n+1, n}$ function at characteristic feet $\{x^{n+1,n}_{j, i_{gl}}\}^{k+1}_{i_{gl}=1}$ located over upstream interval $I^{n+1,n}_j$ in Step 1.1a, see Figure 1(b).
		
	\noindent
	{\bf Step 1.1b: Integrating Term I by summation over sub-intervals.}
		From Figure 1(a), we can see that there are two intersections $I^{n+1, n}_{j,1}=[x_{\jL}^{n+1,n},x_{\jL}]$ and $I^{n+1, n}_{j,2}=[x_{\jL},x_{\jR}^{n+1,n}]$ between $I^{n+1, n}_j$ and the background element $I_j$. In general, $I^{n+1, n}_j = \bigcup\limits_lI^{n+1, n}_{j,l}$
	where $l$ is the index for sub-intervals. Term I is approximated by
	\begin{align} \label{eq3:2:1}
	\int_{I^{n+1, n}_j} u^n \psi^{n+1,n} \ dx \approx \sum_l \int_{I^{n+1, n}_{j,l}} u^n \Psi^{n+1,n}_\star \ dx.
	\end{align}
	On each of these subintervals $I^{n+1, n}_{j,l}$, $u^n \Psi^{n+1, n}_\star$ is continuous and its integration can be approximated by quadrature rules. Notice that $u^n$ is discontinuous across cell boundaries.

%-----------------------------

% Evaluation of TermII

\noindent
{\bf Step 1.2: Evaluation of Term II of \eqref{eq3:3:2} along characteristics intervals.}
There are two technical components involved in this step: one is an LDG approximation to the second order derivative term $u_{xx}$ together with a proper evaluation of $\int_{\widetilde{I}_j^{n+1,n}(t)} [\epsilon u_{xx}+g] \psi \ dx$; the other is the high order temporal discretization by a RK method for
\beq \label{eq3:3:2_df}
	\f{d}{dt} \int_{\widetilde{I}_j^{n+1,n}(t)} u \psi\ dx	=	\int_{\widetilde{I}_j^{n+1,n}(t)}  [\epsilon u_{xx}+g] \psi \ dx,
\eeq
which is the time differential form of eq.~\eqref{eq3:3:2}.
We will first discuss the evaluation of Term II coupled with a simple backward Euler time discretization. Then we will extend the idea to high order time integration by employing diagonally implicit RK (DIRK) methods. The diagonally implicit property allows one to solve a linear system for the current RK stage only, greatly reducing computational complexity and cost.
	
	%LDG
	
	\noindent
	{\bf Step 1.2a: LDG approximation of $u_{xx}$ \cite{cockburn1998local}.} We use the LDG formulation to seek $p\in V^k_h$ approximating $u_{xx}$. In particular, $p=u_{xx}$ can be rewritten as a first order system
	\beq \label{eq2:11}
		\begin{cases}
		p=q_x,\\
		q=u_x.
		\end{cases}
	\eeq
	
Then, we seek $p,q \in V_h^k$ such that, for all test functions $v,w\in P^k(I_j)$,
	\begin{subequations} \label{eq2:12}	
		\begin{align}
		(p,v)_{I_j}		&=	\hat{q}_{\jR}v^{-}_{\jR}-\hat{q}_{\jL}v^{+}_{\jL}-(q,v_x)_{I_j},	\label{eq2:12a}\\
		(q,w)_{I_j} 	&=	\hat{u}_{\jR}w^{-}_{\jR}-\hat{u}_{\jL}w^{+}_{\jL}-(u,w_x)_{I_j},	\label{eq2:12b}
		\end{align}
	\end{subequations}
	where $(\cdot,\cdot)_{I_j}$ stands for the $L^2$ inner product on interval $I_j$, and
 	$\hat{\cdot}$ denotes the numerical fluxes defined at the cell interfaces, which are taken as the alternating fluxes for stability consideration
	\begin{align} \label{eq2:13}
	\hat{q} = q^{-}, 		\quad    \hat{u} = u^{+};
	\quad \text{or}				 \quad
	\hat{q} = q^{+},	 	\quad    \hat{u} = u^{-}.
	\end{align}

	Notice that $q$ can be solved explicitly in terms of $u$ from $(\ref{eq2:12b})$; and also $p$ can 	be solved explicitly from $q$ from \eqref{eq2:12a}. In short, $p = u_{xx}$ can be computed locally by using $u$ from three nearby elements, namely $I_{j-1}$, $I_j$ and $I_{j+1}$.

	%Time discretization
	
	\noindent
	{\bf Step 1.2b. DIRK methods for accurate evaluations of the time integral.}
		
	We start from a first order backward Euler time discretization of \eqref{eq3:3:2_df}:
	\begin{align} \label{eq3:4}
	(u^{n+1},\Psi)_{I_j}-(u^n,\psi^{n+1, n})_{I^{n+1, n}_j}	=	\Dt \left( \epsilon u^{n+1}_{xx}+g^{n+1},\Psi \right)_{I_j}.
	\end{align}
	After rearranging the terms in eq. $(\ref{eq3:4})$, we obtain
	\begin{align} \label{eq3:5}
	(u^{n+1},\Psi)_{I_j}-\epsilon \Dt (u^{n+1}_{xx},\Psi)_{I_j}	=	( u^n,\psi^{n+1,n})_{I^{n+1, n}_j}+	\Dt(g^{n+1},\Psi )_{I_j}.
	\end{align}
	For notational simplicity of the presentation, above we let $\left(\epsilon u_{xx},\Psi \right)_{I_j}$ represent the LDG discretization of the diffusion term, without writing out all the flux and volume integral terms from integration-by-part in an LDG formulation.

	With the test function $\Psi$ going through all basis functions in $V^k_h$, we can formulate a linear system for degrees of freedom (i.e., the coefficients of the basis) of ${\bf u}^{n+1}$ as
	\beq \label{eq: lsystem}
	B_1 {\bf u}^{n+1} = {\bf f}_1,
	\eeq
	which can be solved by an iterative method, e.g. GMRES.
	Here the matrix $B_1 = I - \epsilon \Dt D_{\Delta}$, where $D_{\Delta}$ comes from an LDG discretization of $u_{xx}$; and ${\bf f}_1$ can be obtained from evaluating right-hand side (RHS) terms of \eqref{eq3:5}. The details in constructing the sparse matrix $B_1$ are provided in the Appendix.

	% DIRK2	
	
	To attain higher order accuracy in time, we propose to employ high order DIRK methods. Here, we demonstrate the scheme with an L-stable, two-stage, second-order DIRK method, termed as DIRK2 \cite{ascher1997implicit} (as in Table \ref{tab3.2}) that involves two stages: $t^{(1)}=t^n+\nu \Dt$ and $t^{(2)}=t^{n+1}$.
	
	\begin{table}[!ht]	
	\begin{center}
	\begin{tabular}{c	|		c	c   c}
	
		$\nu$	&			&	$\nu$ 	&	0	\\
		1		&			&	$1-\nu$ 	& $\nu$	\\
		\hline
 				&			&	$1-\nu$ 	& $\nu$
	\end{tabular}
	, \qquad $\nu = 1-\sqrt{2}/2$.
	\end{center}
	\caption{DIRK2.}
	\label{tab3.2}
	\end{table}

 For the convenience of our presentations for DIRK discretization along characteristics, we introduce the following notations
\beq
\label{eq: 1d_notations}
\psi^{\tau_2, \tau}(x), \qquad I^{\tau_2, \tau}_j.
\eeq

%The following lines for 1D schematic illustration of the formulation

\bit

\item[(i)]  $\psi^{\tau_2, \tau}(x)$ denotes the solution $\psi(x, \tau)$ satisfying the final value problem \eqref{eq3:3a}.  Here $\tau_2$ and $\tau$ may refer to intermediate RK stages in a DIRK discretization. Assuming DIRK has $s$ stage
$t^n <t^{(1)}<\cdots<t^{(s)} = t^{n+1},$ $\psi^{t^{( ii )}, t^{(jj)}}(x)$ with $1\le jj \le ii \le s$, denotes the function $\psi(x)$ at $t=t^{(ii)}$ satisfying the adjoint problem \eqref{eq3:3a}
with the final value $\psi(x, t^{( ii )}) = \Psi(x) \in V^k_h$.
For notational simplicity, we let  $\psi^{( ii ), ( jj )}\doteq \psi^{t^{( ii )}, t^{(jj )}}(x)$.

\item[(ii)] $I^{\tau_2, \tau}_j =  [x^{\tau_2, \tau}_{\jL}, x^{\tau_2, \tau}_{\jR}]$ with
$x^{\tau_2, \tau}_{j\pm\frac12}$ being the solution to eq.~\eqref{eq: char_inter} at time $\tau$ with $\widetilde{x}_{j \pm \f12} (\tau_2) = x_{j\pm\f12}$.
For example, $I^{t^{( ii )}, t^{( jj )}}_j$ $\doteq$ $[\widetilde{x}_{\jL}(t^{( jj )}),$ $ \widetilde{x}_{\jR}(t^{( jj )})]$ with $\widetilde{x}_{j\pm\frac12}(t)$ satisfying \eqref{eq: char_inter} and the final value $\widetilde{x}_{j \pm \f12} (t^{( ii )}) = x_{j\pm\f12}$, respectively.
For simplicity, we let $I^{(ii ),(jj ) }_j \doteq I^{t^{( ii )}, t^{( jj )}}_j$.

\eit

Following the above notations, the proposed SLDG-LDG scheme when coupled with a DIRK2 method (see Table \ref{tab3.2}) along characteristics curves can be implemented as below.
		
		\bit	
		
		% First immediate time stage of DIRK2
	
		\item[(i)]  In the \textbf{first time stage} $\tau_2 = t^{(1)}$, as shown in Figure \ref{schematic_dirk2(a)}, for each Eulerian background cell $I_j$, we solve the numerical solution $u^{(1)}\in V_h^k$ at intermediate stage $t^{(1)}$ from the following formulation
		\beq \label{eq3:7}
		(u^{(1)},\Psi)_{I_j}-(u^n,\psi^{(1), n})_{I_j^{(1), n}}		=	\Dt \cdot \nu \left( \epsilon u^{(1)}_{xx}	+ g^{(1)},\Psi \right)_{I_j}.
		\eeq
	    Note that the formulation is equivalent to applying a first order backward Euler method with $\nu \Dt$.
		Implementation-wise, $(\ref{eq3:7})$ can be written as
		$B_2 {\bf u}^{(1)}	=	{\bf f}_2$,
		where $B_2$ and $\bf{f}_2$ can be collected in a similar fashion as those for matrix $B_1$ and vector ${\bf f}_1$ in eq. \eqref{eq: lsystem}.
		
		% Second immediate time stage of DIRK2

\definecolor{officegreen}{rgb}{0.0, 0.5, 0.0}
\definecolor{darkpastelgreen}{rgb}{0.01, 0.75, 0.24}
\definecolor{burntorange}{rgb}{0.8, 0.33, 0.0}
\definecolor{armygreen}{rgb}{0.29, 0.33, 0.13}
\definecolor{applegreen}{rgb}{0.55, 0.71, 0.0}
\definecolor{ao}{rgb}{0.0, 0.5, 0.0}
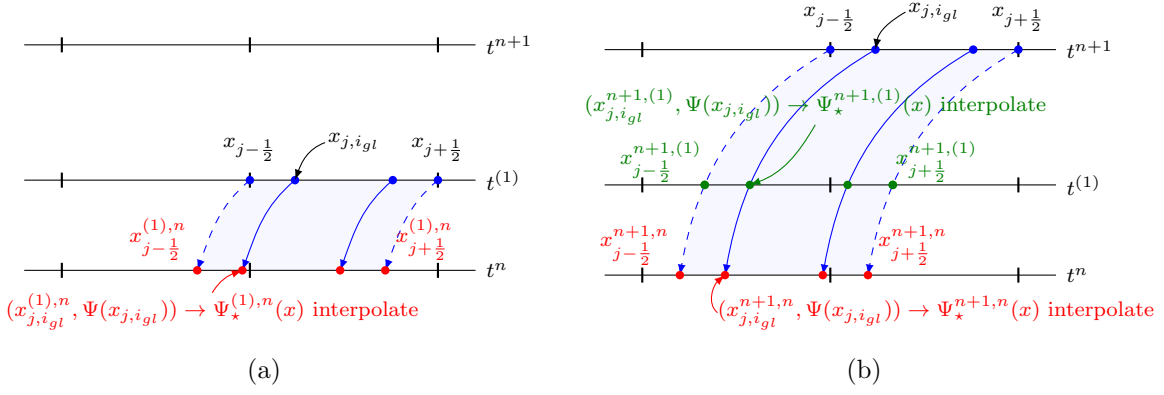
\begin{figure}[h!]
\centering
\subfigure[]{
\label{schematic_dirk2(a)}
\begin{tikzpicture}[scale=1.0]
%-------------------------------------------
% draw the integration region as a background
    \draw[white,fill=blue!3] (0,1.2) to[out=220,in=70] (0-0.7,0)  -- (2.5-0.7,0) to[out=70,in=220] (2.5, 1.2)
      -- cycle;
%-------------------------------------------
% draw two stage, t^n and t^{n+1}
    \draw[black]                 (-3,0) node[left] { } -- (3,0)
                                        node[right]{\scriptsize$t^{n}$};

    \draw[black]                 (-3,1.2) node[left] { } -- (3,1.2)
                                        node[right]{\scriptsize$t^{(1)}$};

    \draw[black]                 (-3,3) node[left] { } -- (3,3)
                                        node[right]{\scriptsize$t^{n+1}$};
%-------------------------------------------
% draw ticks
% 3 evenly spaced ticks at t^{n+1} using a length 2.5
    \draw[thick] (0, 3-0.1) -- (0, 3+0.1) node[above] { };
    \draw[thick] (2.5, 3-0.1) -- (2.5, 3+0.1) node[above] { };
    \draw[thick] (-2.5, 3-0.1) -- (-2.5, 3+0.1) node[above] {};
% 3 evenly spaced ticks at t^{n+1} using a length 2.5
    \draw[thick] (0, 0-0.1) -- (0, 0+0.1) node[above] { };
    \draw[thick] (2.5, 0-0.1) -- (2.5, 0+0.1) node[above] { };
    \draw[thick] (-2.5, 0-0.1) -- (-2.5, 0+0.1) node[above] { };
% 3 evenly spaced ticks at t^{(1)} using a length 2.5
    \draw[thick] (0, 1.2-0.1) -- (0, 1.2+0.1) node[above] { \scriptsize$x_{j-\frac12}$};
    \draw[thick] (2.5, 1.2-0.1) -- (2.5, 1.2+0.1) node[above] {\scriptsize$x_{j+\frac12}$  };
    \draw[thick] (-2.5, 1.2-0.1) -- (-2.5, 1.2+0.1) node[above] { };
%-------------------------------------------
% draw two characteristic lines
 \draw[-latex,dashed,blue]( 0, 1.2 )node[left,scale=1.3]{$$}
        to[out=220,in=70] ( 0-0.7 ,0) node[below=2pt] { };
 \draw[-latex,dashed,blue]( 2.5, 1.2 )node[left,scale=1.3]{$$}
        to[out=220,in=70] ( 2.5-0.7 ,0) node[below=2pt] { };
%-------------------------------------------
% draw nodes
 \fill [blue] ( 0, 1.2 ) circle (1.6pt) node[right] {};
 \fill [blue] ( 2.5, 1.2 ) circle (1.6pt) node[right] {};

 \fill [red] ( 0-0.7, 0 ) circle (1.6pt) node[above left=-0.1 ] {\scriptsize$x_{j-\frac12}^{(1),n}$};
 \fill [red] ( 2.5-0.7, 0 ) circle (1.6pt) node[above right=-0.1] {\scriptsize$x_{j+\frac12}^{(1),n}$};
%-------------------------------------------
 \draw[-latex,red ]( - 0.5  , -0.3 )node[below=-4pt,scale=1.0]{\scriptsize$(x_{j,i_{gl}}^{(1),n}, \Psi(x_{j,i_{gl}})) \rightarrow \Psi_\star^{(1),n}(x) \text{ interpolate} $ }
        to[out= 60,in=200] ( 0+0.6-0.7, 0) node[below=2pt] { };
%-------------------------------------------
 \draw[-latex, blue]( 0+0.6, 1.2 )node[left,scale=1.3]{$$}
        to[out=220,in=70] ( 0-0.7+0.6 ,0) node[below=2pt] { };
 \draw[-latex, blue]( 2.5-0.6, 1.2 )node[left,scale=1.3]{$$}
        to[out=220,in=70] ( 2.5-0.7-0.6 ,0) node[below=2pt] { };
% draw lobatto nodes
 \fill [blue] ( 0+0.6, 1.2 ) circle (1.6pt) node[right] {};
 \fill [blue] ( 2.5-0.6, 1.2 ) circle (1.6pt) node[right] {};

 \fill [red] ( -0.7+0.6, 0 ) circle (1.6pt) node[above left=-0.1 ] { };
 \fill [red] ( 2.5-0.7 -0.6, 0 ) circle (1.6pt) node[above right=-0.1] { };

 \draw[-latex ]( 0+1, 1.7 )node[right=-3pt,scale=1.0]{\scriptsize$x_{j,i_{gl}}$}
        to[out=200,in=70] ( 0+0.6, 1.2) node[below=2pt] { };

\end{tikzpicture}
}
%-------------------------------------------
\subfigure[]{
\label{schematic_dirk2(b)}
\begin{tikzpicture}[scale=1.0]
%-------------------------------------------
% draw the integration region as a background
    \draw[white,fill=blue!3] (0,3) to[out=210,in=80] (0-2.,0)  -- (2.5-2,0) to[out=80,in=210] (2.5 ,3)
      -- cycle;
%-------------------------------------------
% draw two stage, t^n and t^{n+1}
    \draw[black]                 (-3,0) node[left] { } -- (3,0)
                                        node[right]{\scriptsize$t^{n}$};

    \draw[black]                 (-3,1.2) node[left] { } -- (3,1.2)
                                        node[right]{\scriptsize$t^{(1)}$};

    \draw[black]                 (-3,3) node[left] { } -- (3,3)
                                        node[right]{\scriptsize$t^{n+1}$};
%-------------------------------------------
% draw ticks
% 3 evenly spaced ticks at t^{n+1} using a length 2.5
    \draw[thick] (0, 3-0.1) -- (0, 3+0.1) node[above] {\scriptsize$x_{j-\frac12}$};
    \draw[thick] (2.5, 3-0.1) -- (2.5, 3+0.1) node[above] {\scriptsize$x_{j+\frac12}$};
    \draw[thick] (-2.5, 3-0.1) -- (-2.5, 3+0.1) node[above] {};
% 3 evenly spaced ticks at t^{n+1} using a length 2.5
    \draw[thick] (0, 0-0.1) -- (0, 0+0.1) node[above] { };
    \draw[thick] (2.5, 0-0.1) -- (2.5, 0+0.1) node[above] { };
    \draw[thick] (-2.5, 0-0.1) -- (-2.5, 0+0.1) node[above] { };
% 3 evenly spaced ticks at t^{(1)} using a length 2.5
    \draw[thick] (0, 1.2-0.1) -- (0, 1.2+0.1) node[above] { };
    \draw[thick] (2.5, 1.2-0.1) -- (2.5, 1.2+0.1) node[above] { };
    \draw[thick] (-2.5, 1.2-0.1) -- (-2.5, 1.2+0.1) node[above] { };
%-------------------------------------------
% draw two characteristic lines
 \draw[-latex,dashed,blue]( 0, 3 )node[left,scale=1.3]{$$}
        to[out=210,in=80] ( 0-2,0) node[below=2pt] { };
 \draw[-latex,dashed,blue]( 2.5, 3 )node[left,scale=1.3]{$$}
        to[out=210,in=80] ( 2.5-2,0) node[below=2pt] { };
%-------------------------------------------
% draw nodes
 \fill [blue] ( 0, 3 ) circle (1.6pt) node[right] {};
 \fill [blue] ( 2.5, 3 ) circle (1.6pt) node[right] {};

 \fill [red] ( 0-2, 0 ) circle (1.6pt) node[above left=-0.1 ] {\scriptsize$x_{j-\frac12}^{n+1,n}$};
 \fill [red] ( 2.5-2, 0 ) circle (1.6pt) node[above right=-0.1] {\scriptsize$x_{j+\frac12}^{n+1,n}$};
%-------------------------------------------
% draw
 \draw[-latex ,blue]( 0+0.6, 3 )node[left,scale=1.3]{$$}
        to[out=210,in=80] ( 0-2+0.6,0) node[below=2pt] { };
 \draw[-latex, blue]( 2.5-0.6, 3 )node[left,scale=1.3]{$$}
        to[out=210,in=80] ( 2.5-2-0.6,0) node[below=2pt] { };
%-------------------------------------------
% draw nodes
 \fill [blue] ( 0+0.6, 3 ) circle (1.6pt) node[right] {};
 \fill [blue] ( 2.5-0.6, 3 ) circle (1.6pt) node[right] {};

 \fill [red] ( 0-2+0.6, 0 ) circle (1.6pt) node[above left=-0.1 ] { };
 \fill [red] ( 2.5-2-0.6, 0 ) circle (1.6pt) node[above right=-0.1] { };
%-------------------------------------------
 \draw[-latex ]( 0+1, 3.5 )node[right=-3pt,scale=1.0]{\scriptsize$x_{j,i_{gl} }$}
        to[out=210,in=80] ( 0+0.6, 3) node[below=2pt] { };
%-------------------------------------------
 \draw[-latex,red]( -1.5 , -0.5 )node[right=-4pt,scale=1.0]{\scriptsize$(x_{j,i_{gl}}^{n+1,n}, \Psi(x_{j,i_{gl} })) \rightarrow \Psi_\star^{n+1,n}(x) \text{ interpolate} $ }
        to[out=140,in=220] ( 0+0.6-2, 0) node[below=2pt] { };

%-------------------------------------------
 \fill [ao] ( 0-1.67+0.6, 1.2 ) circle (1.6pt) node[above left  ] { };
 \fill [ao] ( 2.5-1.67-0.6, 1.2 ) circle (1.6pt) node[above right ] { };

 \fill [ao] ( 0-1.67 , 1.2 ) circle (1.6pt) node[above left=-0.1 ] { };
 \fill [ao] ( 2.5-1.67, 1.2 ) circle (1.6pt) node[above right=-0.1] { };

  \fill [ao] ( 0-1.67 , 1.2 ) circle (1.6pt) node[ above left=-3pt  ] {\scriptsize$x_{j-\frac12}^{n+1,(1)}$};
 \fill [ao] ( 2.5-1.67, 1.2 ) circle (1.6pt) node[ above right=-3pt] {\scriptsize$x_{j+\frac12}^{n+1,(1)}$};
%-------------------------------------------
 \draw[-latex, ao]( -0.2  , 2 )node[above= -4pt,scale=1.0]{\scriptsize$(x_{j,i_{gl}}^{n+1,(1)}, \Psi(x_{j,i_{gl} })) \rightarrow \Psi_\star^{n+1,(1)}(x) \text{ interpolate} $ }
        to[out=240,in=20] ( 0+0.6-1.67, 1.2) node[below=2pt] { };
\end{tikzpicture}
}
\caption{Schematic illustration of 1D SLDG-LDG formulation coupled with DIRK2. Left: First time stage $t^{(1)}$. Right: Second time stage $t^{n+1}$.}
\label{schematic_dirk2}
\end{figure}

		\item[(ii)] In the \textbf{second time stage} $\tau_2 = t^{n+1}$, as shown in Figure \ref{schematic_dirk2(b)},we have the following formulation
		\begin{multline} \label{eq3:11}
		(u^{n+1},\Psi)_{I_j}-(u^n,\psi^{n+1, n})_{I_j^{n+1, n}}		=	\\
		\Dt \left[(1-\nu) \cdot \left( \epsilon u^{(1)}_{xx}+g^{(1)},\psi^{n+1, (1)} \right)_{I^{n+1, (1)}_j}+
		\nu \cdot \left( \epsilon u^{n+1}_{xx}+g^{n+1},\Psi \right)_{I_j}\right].
		\end{multline}
		Notice that $\psi^{n+1, n}, \psi^{n+1, (1)}$ are in general not polynomials, yet can be well approximated by polynomials as in Step 1.1b. Reorganizing terms in \eqref{eq3:11} gives
		\begin{multline} \label{eq3:13}
		(u^{n+1},\Psi)_{I_j}-\Dt \cdot \nu \epsilon( u^{n+1}_{xx},\Psi)_{I_j}	=
		(u^n,\psi^{n+1, n})_{I_j^{n+1, n}}+
		\Dt \cdot (1-\nu)\epsilon(u^{(1)}_{xx},\psi^{n+1, (1)})_{I^{n+1, (1)}_j}\\
		+\Dt \cdot (1-\nu) \left( g^{(1)},\psi^{n+1, (1)} \right)_{I^{n+1, (1)}_j}+
		\Dt \cdot \nu \left( g^{n+1},\Psi \right)_{I_j}.
		\end{multline}
		
		Notice that the first term on RHS of \eqref{eq3:13} can be evaluated as in Step 1.1; the second term on RHS of \eqref{eq3:13} can be evaluated by first computing $u^{(1)}_{xx}$ in an LDG fashion with $u^{(1)}$ given from the first stage of RK computation, and then applying {\em the same procedure as in Step 1.1} to evaluate $(u^{(1)}_{xx},\psi^{n+1, (1)})_{I^{n+1, (1)}_j}$; the latter two terms involving $g$ can be directly evaluated by quadrature rules. Implementation-wise, \eqref{eq3:13} can be written as
		$B_3 {\bf u}^{n+1}={\bf f}_3$ with $B_3$ the same matrix as $B_2$ in the first time stage.
		
		\eit

% General DIRK time discretization

\begin{rem} The above procedure is for a two-stage second-order DIRK discretization of diffusion and source terms. Such a procedure can be generalized to any DIRK methods. For some high order DIRK discretization methods we use for the numerical experiment, the associated Butcher tableaus are provided in the Appendix including the L-stable, three-stage, third-order DIRK method in Table $\ref{tab_dirk3}$ \cite{calvo2001linearly}, the L-stable, five-stage, fourth-order method in Table $\ref{tab_dirk4}$ \cite{wanner1991solving}. Notice that, we use the SLDG method for the convection term and an implicit discretization for the diffusion and source terms; thus the time stepping size allowed could be much larger than that of an explicit Eulerian RKDG method.
\end{rem}
\begin{rem}
All DIRK time discretization methods we employ in the paper have the property that $a_{ii}\neq0$, $\forall i=1, \cdots s$ and the method are stiffly accurate; these properties are important for numerical stability.
\end{rem}

%-----------------------------------------------------------------------------------------------------------------------------------------------------------------------------------------------------

\subsection{Scheme formulation: 2D case}
\label{2D_form}

In this subsection, we generalize the above 1D SLDG-LDG scheme for solving the following 2D problem
\beq \label{eq2:4:2d}
u_t +(a(x,y,t) u)_x +(b(x,y,t) u)_y 	=	 \epsilon \Delta u + g.
\eeq
We begin with a partition of the 2D domain as $\Omega = \{E_{j}\}^J_{j=1}$. The numerical solutions and test functions belong to the finite dimensional piecewise approximation space
\beq \label{eq2:4:space}
V^k_h = \{v_h:v_h|_{E_j} \in P^k(E_j), j=1,2,\cdots,J \}
\eeq
where $P^k(E_j)$ denotes the set of polynomials of degree at most $k$ over each element $E_j$.

% Formulation

% Adjoint problem

Similar to the 1D case and the strategy in \cite{cai2017high}, we consider the adjoint problem for the test function $\psi=\psi(x,y,t)$ satisfying
\beq \label{eq2:4:ad}
\psi_t+a(x,y,t)\psi_x+b(x,y,t)\psi_y		=	0, \quad t \in [\tau_1, \tau_2].
\eeq
with $\psi(x, y, \tau_2) = \Psi(x, y) \in V^k_h$.
A similar observation as in the 1D case is that the solution to \eqref{eq2:4:ad} stays constant along the characteristic curves governed by
\[
\f{d \widetilde{x} (t)}{dt}=a(\widetilde{x}(t),\widetilde{y}(t),t), \quad \f{d \widetilde{y} (t)}{dt}=b(\widetilde{x} (t), \widetilde{y} (t),t).
\]
Let $\widetilde{E}_j(t)$ be the dynamic moving cell bounded by characteristics curves emanating from the edges of Eulerian cell $E_j$ at $t^{n+1}$ and $E^{n+1, n}_j$ be the upstream cell as $\widetilde{E}_j(t=t^n)$, see Figure $\ref{schematic_2d_a}$. A 2D generalization of Proposition $\ref{prop_1D}$ is established in the following.

% 2D proposition

\begin{prop}\label{prop_2D}

Consider the 2D problem \eqref{eq2:4:2d} and the adjoint problem \eqref{eq2:4:ad} for the test function $\psi$,
then the following identity holds
\beq \label{eq2:4:prop}
\iint_{E_j}u^{n+1}\Psi \ dx\ dy -\underbrace{\iint_{E^{n+1, n}_j} u^n\psi^{n+1, n} \ dx\ dy}_{\text{Term I}} = \underbrace{\int^{t^{n+1}}_{t^n} \iint_{\widetilde{E}_j(t)} [\epsilon \Delta u+g] \psi \ dx \ dy\ dt}_{\text{Term II}},
\eeq
where $\psi^{n+1, n}$ is the solution to the adjoint problem \eqref{eq2:4:ad} at $t^n$ with $\psi(x,y,t^{n+1}) = \Psi \in V^k_h$.
\end{prop}

Similar to the 1D case, the update of $u^{n+1}\in V^k_h$ depends on proper evaluations of Term I and Term II of eq. \eqref{eq2:4:prop}. We again refer to \cite{cai2017high} for detailed procedures of evaluating Term I and only summarize main steps below. The computation of Term II consists of two parts: the first part is approximating $\Delta u$ by using 2D LDG spatial discretization, and the second part is high order time integration over the dynamic moving cell $\tilde{E}_j(t)$ along the characteristics. These two parts share the same spirit with Step 1.2 in Subsection $\ref{1D_form}$. Below we outline the main procedures for a 2D problem.

% 2D schematic illustration of the formulation with P^1

\begin{figure}[htbp]\small
\centering
\subfigure[]{
\label{schematic_2d_a}

\begin{tikzpicture}

\draw[black,thin] (0,0.5) node[left] {} -- (5.5,0.5) node[right]{};
\draw[black,thin] (0,2.) node[left] {} -- (5.5,2) node[right]{};
\draw[black,thin] (0,3.5) node[left] {} -- (5.5,3.5) node[right]{};
\draw[black,thin] (0,5 ) node[left] {} -- (5.5,5) node[right]{};

\draw[black,thin] (0.5,0) node[left] {} -- (0.5,5.5) node[right]{};
\draw[black,thin] (2,0) node[left] {} -- (2,5.5) node[right]{};
\draw[black,thin] (3.5,0) node[left] {} -- (3.5,5.5) node[right]{};
\draw[black,thin] (5,0) node[left] {} -- (5,5.5) node[right]{};

\fill [blue] (3.5,3.5) circle (2pt) node[] {};
\fill [blue] (5,3.5) circle (2pt) node[] {};
\fill [blue] (3.5,5) circle (2pt) node[below right] {$E_j$} node[above left] {$v_4$};
\fill [blue] (5,5) circle (2pt) node[] {};

\draw[thick,blue] (3.5,3.5) node[left] {} -- (3.5,5) node[right]{};
\draw[thick,blue] (3.5,3.5) node[left] {} -- (5,3.5) node[right]{};
\draw[thick,blue] (3.5,5) node[left] {} -- (5,5) node[right]{};
\draw[thick,blue] (5,3.5) node[left] {} -- (5,5) node[right]{};

\fill [red] (1.,1) circle (2pt) node[above right,black] {};
\fill [red] (3,1) circle (2pt) node[] {};
\fill [red] (1,2.5) circle (2pt) node[below right] {$E_j^{n+1, n}$} node[above left] {$v_4^{n+1,n}$};
\fill [red] (2.5,2.5) circle (2pt) node[] {};

\draw[-latex,dashed](3.5,5)node[right,scale=1.0]{} to[out=240,in=70] (1,2.50) node[] {};

\draw (0.5+0.01,2-0.01) node[fill=white,below right] {$E_l$};

\draw [red,thick] (1,1)node[right,scale=1.0]{} to[out=20,in=150] (2,0.7) node[] {};

\draw [red,thick] (2,0.7)node[right,scale=1.0]{} to[out=330,in=240] (3,1) node[] {};

\draw [red,thick] (1,2.5)node[right,scale=1.0]{} to[out=310,in=90] (1.1,2) node[] {};

\draw [red,thick] (1.1,2)node[right,scale=1.0]{} to[out=270,in=80] (1,1) node[] {};

\draw [red,thick] (1,2.5)node[right,scale=1.0]{} to[out=10,in=180] (2.5,2.5) node[] {};

\draw [red,thick] (3,1)node[right,scale=1.0]{} to[out=80,in=280] (2.5,2.5) node[] {};

\end{tikzpicture}
}
\subfigure[]{
\label{schematic_2d_b}

\begin{tikzpicture}[scale = 1.3]

\draw[black,thin] (0,0.5) node[left] {} -- (4,0.5) node[right]{};
\draw[black,thin] (0,2.) node[left] {} -- (4,2) node[right]{};
\draw[black,thin] (0,3.5) node[left] {} -- (4,3.5) node[right]{};
\draw[black,thin] (0.5,0) node[left] {} -- (0.5,4) node[right]{};
\draw[black,thin] (2,0) node[left] {} -- (2,4) node[right]{};
\draw[black,thin] (3.5,0) node[left] {} -- (3.5,4) node[right]{};

\fill [red] (1.,1) circle (2pt) node[above right,black] {$E_{j,l}^{n+1, n}$};
\fill [red] (3,1) circle (2pt) node[] {};
\fill [red] (1,2.5) circle (2pt) node[below right] {$E_j^{n+1, n}$} node[above left] {};
\fill [red] (2.5,2.5) circle (2pt) node[] {};

\draw (0.5+0.01,2-0.01) node[fill=white,below right] {};

\draw [red,thick] (1,1)node[right,scale=1.0]{} to[out=20,in=150] (2,0.7) node[] {};

\draw [red,thick] (2,0.7)node[right,scale=1.0]{} to[out=330,in=240] (3,1) node[] {};

\draw [red,thick] (1,2.5)node[right,scale=1.0]{} to[out=310,in=90] (1.1,2) node[] {};

\draw [red,thick] (1.1,2)node[right,scale=1.0]{} to[out=270,in=80] (1,1) node[] {};

\draw [red,thick] (1,2.5)node[right,scale=1.0]{} to[out=10,in=180] (2.5,2.5) node[] {};

\draw [red,thick] (3,1)node[right,scale=1.0]{} to[out=80,in=280] (2.5,2.5) node[] {};

\draw (0.5+0.01,2-0.01) node[fill=white,below right] {$E_l$};

\draw[-latex,ultra thick] (1,1)node[right,scale=1.0]{} to  (2,1) node[] {};

\draw[-latex,ultra thick]  (2,1)node[right,scale=1.0]{} to (2,2) node[] {};

\draw[-latex,ultra thick]  (2,2)node[right,scale=1.0]{} to (1,2) node[] {};

\draw[-latex,ultra thick]   (1,2)node[right,scale=1.0]{} to (1,1) node[] {};

\draw [thick] (1,1)-- (3,1) node[] {};
\draw [thick] (1,2.5) -- (1,1) node[] {};
\draw [thick] (1,2.5)--(2.5,2.5) node[] {};
\draw [thick] (3,1)--(2.5,2.5) node[] {};

\end{tikzpicture}
}
\caption{Schematic illustration of the SLDG formulation with $P^1$ polynomial spaces in 2D. Left: upstream cell $E^{n+1, n}_j$. Right: Quadrilateral upstream cells.}
\label{schematic_2d}
\end{figure}
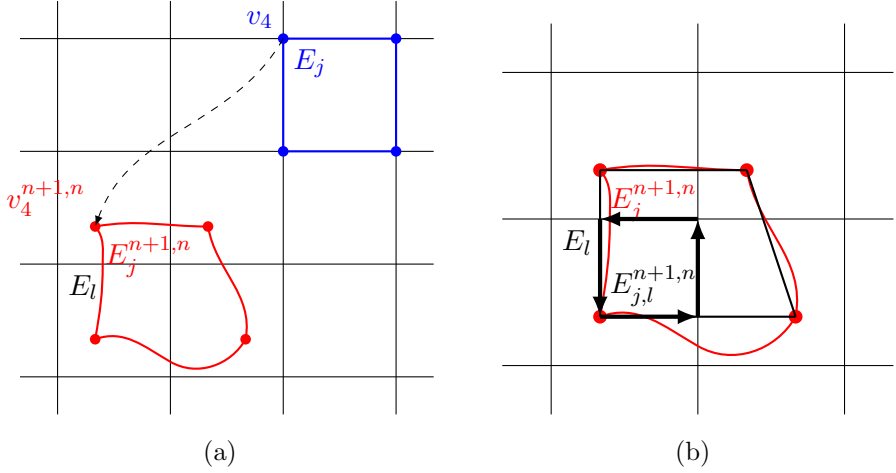

% Evaluations of TermI and TermII

% Evaluation of TermI

\noindent
{\bf Step 2.1: Evaluation of Term I of \eqref{eq2:4:prop}.} It's worth noting that when the velocity field is space and time dependent, the upstream cell $E^{n+1, n}_j$ might not be of quadrilateral shape. When they are approximated by a quadrilateral or a quadratic-curved (QC) quadrilateral, second or third-order spatial accuracy can be achieved, respectively. Here, as an example, we discuss the formulation with $P^1$ polynomial spaces.

	\noindent	
	{\bf Step 2.1a: Characteristic tracing.}
	Locate four vertices $\{v^{n+1,n}_q\}^4_{q=1} $ of upstream cell $E^{n+1, n}_j$ at $t^n$
	by solving the final value problem
	\beq \label{char_2d}
	\begin{cases}
	\f{d \widetilde{x} (t)}{dt}	=	a( \widetilde{x} (t), \widetilde{y} (t),t), 	\quad \widetilde{x} (t^{n+1})		=	x_{v_q},\\
	\f{d \widetilde{y} (t)}{dt}	=	b( \widetilde{x} (t), \widetilde{y} (t),t), 	\quad \widetilde{y} (t^{n+1})		=	y_{v_q}
	\end{cases}
	\eeq
	by high order numerical integrators the same way as in the 1D case, and $\{v_q\}^4_{q=1}$ with coordinates $(x_{v_q}, y_{v_q})$ are the four vertices of $E_j$.
	
	\noindent
	{\bf Step 2.1b: Reconstruction of the test function $\psi^{n+1,n}$ and decomposition of Term I.}
	It is known that $\psi(x,y,t)$ with adjoint problem \eqref{eq2:4:ad} stays constant along characteristics,
	\beq \label{char_2d:2}
	\psi^{n+1, n} (x(v^{n+1,n}_q),y(v^{n+1,n}_q))	=	\Psi(x(v_q),y(v_q)),	\quad q = 1,\cdots,4.
	\eeq
	We can reconstruct $P^k$ $ (k=1)$ polynomial $\Psi^{n+1, n}_\star(x)$ to approximate $\psi^{n+1, n}$ by a least-square strategy.
	Let $E^{n+1, n}_{j,l}$ be intersections between the upstream cell $E^{n+1, n}_j$ and the background cell $E_l$, see in Figure $\ref{schematic_2d_b}$. In general, $E^{n+1, n}_j = \bigcup\limits_{l \in \epsilon^{n+1, n}_j} E^{n+1, n}_{j,l}$ where $\epsilon^{n+1, n}_j \doteq \{l\vert E^{n+1, n}_{j,l}\neq \emptyset \ \text{where}\ E^{n+1, n}_{j,l} \doteq E^{n+1, n}_j\cap E_l \}$, and then Term I in $(\ref{eq2:4:prop})$ can be approximated with
	\beq \label{eq2:sum_2d}
	\iint_{E^{n+1, n}_j} u^n\psi^{n+1, n} \ dx\ dy \approx \sum\limits_{l \in \epsilon^{n+1, n}_j}\iint_{E^{n+1, n}_{j,l}} u^n\Psi^{n+1, n}_\star(x, y) \ dx\ dy.
	\eeq
	\noindent
	{\bf Step 2.1c: Line integral evaluation.}
	To evaluate the area integral $\iint_{E^{n+1, n}_{j,l}} u^n \Psi^{n+1, n}_\star(x, y) \ dx\ dy$, we can introduce two auxiliary functions $P(x,y)$ and $Q(x,y)$ satisfying
	\begin{equation*}
	-\f{\partial P}{\partial y}+\f{\partial Q}{\partial x}	=	u^n \Psi^{n+1, n}_\star(x, y).
	\end{equation*}
	Due to the Green's theorem, the area integral can be converted into the line integrals as
	\beq \label{eq2_lint_qd}
	\iint_{E^{n+1,n}_{j,l}} u^n \Psi^{n+1, n}_\star(x, y) \ dx\ dy		=	\oiint_{\partial E^{n+1, n}_{j,l}} P\ dx+Q\ dy,
	\eeq
	where quadrature rules can be directly applied along $\partial E^{n+1, n}_{j,l}$. This evaluation procedure is the same as in \cite{lauritzen2010conservative}.

%Extensions to $P^2$ with quadratic curved quadrilateral.

\begin{rem}
Applying quadrilateral approximation to $P^2$ polynomial spaces will restrict us with the second-order accuracy in a general setting. This motivates us to use QC quadrilateral approximation to the upstream cells for higher order accuracy. There are two additional key steps. First, locate nine upstream points $\{v^{n+1,n}_q\}^9_{q=1}$ belonging to the upstream cell $E^{n+1, n}_j$, see Figure $\ref{fig:schematic_2d_p2_a}$, by solving \eqref{char_2d} with final values $\{v_q\}^9_{q=1}$ (nine uniformly distributed points at $E_j$). Second, approximate each side of the upstream cell with a quadratic curve by a parameterization and evaluation of the area integral through the line integrals. For more details, we refer to \cite{cai2017high}.
\end{rem}
\begin{figure}[htbp]
\centering
\subfigure[]{
\label{fig:schematic_2d_p2_a}
\begin{tikzpicture}

%background grids

\draw[black,thin] (0,0.5) node[left] {} -- (5.5,0.5) node[right]{};

\draw[black,thin] (0,2.) node[left] {} -- (5.5,2) node[right]{};

\draw[black,thin] (0,3.5) node[left] {} -- (5.5,3.5) node[right]{};

\draw[black,thin] (0,5 ) node[left] {} -- (5.5,5) node[right]{};

\draw[black,thin] (0.5,0) node[left] {} -- (0.5,5.5) node[right]{};

\draw[black,thin] (2,0) node[left] {} -- (2,5.5) node[right]{};

\draw[black,thin] (3.5,0) node[left] {} -- (3.5,5.5) node[right]{};

\draw[black,thin] (5,0) node[left] {} -- (5,5.5) node[right]{};

% current Eulerian cells

\fill [blue] (3.5,3.5) circle (2pt) node[] {};

\fill [blue] (5,3.5) circle (2pt) node[] {};

\fill [blue] (3.5,5) circle (2pt) node[below right] {$E_j$} node[above left] {$v_7$};

\fill [blue] (5,5) circle (2pt) node[] {};

\fill [blue] (3.5,4.25) circle (2pt) node[] {};

\fill [blue] (5,4.25) circle (2pt) node[] {};

\fill [blue] (4.25,4.25) circle (2pt) node[] {};

\fill [blue] (4.25, 3.5) circle (2pt) node[] {};

\fill [blue] (4.25,5) circle (2pt) node[] {};

% connecting vertices

\draw[thick,blue] (3.5,3.5) node[left] {} -- (3.5,5) node[right]{};

\draw[thick,blue] (3.5,3.5) node[left] {} -- (5,3.5) node[right]{};

\draw[thick,blue] (3.5,5) node[left] {} -- (5,5) node[right]{};

\draw[thick,blue] (5,3.5) node[left] {} -- (5,5) node[right]{};

% upstream cells

\fill [red] (1.,1) circle (2pt) node[above right,black] {$E_{j,l}^{n+1, n}$};

\fill [red] (3,1) circle (2pt) node[] {};

\fill [red] (1,2.5) circle (2pt) node[below right] {$E_j^{n+1, n}$} node[above left] {$v_7^{n+1,n}$};

\fill [red] (2.5,2.5) circle (2pt) node[] {};

\draw[-latex,dashed](3.5,5)node[right,scale=1.0]{} to[out=240,in=70] (1,2.50) node[] {};

\draw (0.5+0.01,2-0.01) node[fill=white,below right] {$E_l$};

\draw [red,thick] (1,1)node[right,scale=1.0]{} to[out=20,in=150] (2,0.7) node[] {};

% connecting

\draw [red,thick] (2,0.7)node[right,scale=1.0]{} to[out=330,in=240] (3,1) node[] {};

\draw [red,thick] (1,2.5)node[right,scale=1.0]{} to[out=310,in=90] (1.1,2) node[] {};

\draw [red,thick] (1.1,2)node[right,scale=1.0]{} to[out=270,in=80] (1,1) node[] {};

\draw [red,thick] (1,2.5)node[right,scale=1.0]{} to[out=10,in=180] (2.5,2.5) node[] {};

\draw [red,thick] (3,1)node[right,scale=1.0]{} to[out=80,in=280] (2.5,2.5) node[] {};

\fill [red] (2,0.7) circle (2pt) node[above right,black] {};

\fill [red] (2,1.7) circle (2pt) node[] {};

\fill [red] (1.9,2.5) circle (2pt) node[below right] {} node[above left] {};

\fill [red] (1.1,1.8) circle (2pt) node[] {};

\fill [red] (2.8,1.8) circle (2pt) node[] {};

% shading the subregion E^{n+1,n}_{j,l}

\draw [red,thin] (1.4,2)node[left]{} -- (1.1,1.7) node[right] {};
\draw [red,thin] (1.6,2)node[left]{} -- (1.1,1.5) node[right] {};
\draw [red,thin] (1.8,2)node[left]{} -- (1.05,1.3) node[right] {};
\draw [red,thin] (2.0,2)node[left]{} -- (1.0,1.0) node[right] {};
\draw [red,thin] (2.0,1.8)node[left]{} -- (1.25,1.05) node[right] {};
\draw [red,thin] (2.0,1.6)node[left]{} -- (1.45,1.0) node[right] {};
\draw [red,thin] (2.0,1.35)node[left]{} -- (1.57,0.93) node[right] {};
\draw [red,thin] (2.0,1.15)node[left]{} -- (1.7,0.86) node[right] {};
\draw [red,thin] (2.0,0.95)node[left]{} -- (1.85,0.8) node[right] {};

\end{tikzpicture}
}
\subfigure[]{
\label{fig:schematic_2d_p2_b}
\begin{tikzpicture}[scale = 1.3]
\draw[black,thin] (0,0.5) node[left] {} -- (4,0.5) node[right]{};

\draw[black,thin] (0,2.) node[left] {} -- (4,2) node[right]{};

\draw[black,thin] (0,3.5) node[left] {} -- (4,3.5) node[right]{};

\draw[black,thin] (0.5,0) node[left] {} -- (0.5,4) node[right]{};
\draw[black,thin] (2,0) node[left] {} -- (2,4) node[right]{};
\draw[black,thin] (3.5,0) node[left] {} -- (3.5,4) node[right]{};

\fill [red] (1.,1) circle (2pt) node[above right,black] {};

\fill [red] (3,1) circle (2pt) node[] {};

\fill [red] (1,2.5) circle (2pt) node[below right] {} node[above left] {};

\fill [red] (2.5,2.5) circle (2pt) node[] {};

\draw (0.5+0.01,2-0.01) node[fill=white,below right] {};

\draw [red,thick] (1,1)node[right,scale=1.0]{} to[out=20,in=150] (2,0.7) node[] {};

\draw [red,thick] (2,0.7)node[right,scale=1.0]{} to[out=330,in=240] (3,1) node[] {};

\draw [red,thick] (1,2.5)node[right,scale=1.0]{} to[out=310,in=90] (1.1,2) node[] {};

\draw [red,thick] (1.1,2)node[right,scale=1.0]{} to[out=270,in=80] (1,1) node[] {};

\draw [red,thick] (1,2.5)node[right,scale=1.0]{} to[out=10,in=180] (2.5,2.5) node[] {};

\draw [red,thick] (3,1)node[right,scale=1.0]{} to[out=80,in=280] (2.5,2.5) node[] {};

\fill [red] (2,0.7) circle (2pt) node[above right,black] {};

\fill [red] (2,1.7) circle (2pt) node[] {};

\fill [red] (1.9,2.5) circle (2pt) node[below right] {} node[above left] {};

\fill [red] (1.1,1.8) circle (2pt) node[] {};

\fill [red] (2.8,1.8) circle (2pt) node[] {};

\draw[-latex,blue,thick](0,1)node[right,scale=1.0]{} to (4,1) node[below] {$\xi$};

\draw[-latex,blue,thick](2,0)node[right,scale=1.0]{} to (2,4) node[right] {$\eta$};

\draw[black, ultra thick] (2,0.7) node[below left =2pt] {$v_2^{n+1,n}$} parabola(1,1)node[above left,scale=1.0]{ $v_1^{n+1,n}$ };

\draw[-latex,black, ultra thick](2,0.7)node[above left,scale=1.0]{  } parabola (3,1) node[below =2pt] {$v_3^{n+1,n}$};

\end{tikzpicture}

}
\caption{Schematic illustration of the SLDG formulation with $P^2$ polynomial spaces in 2D. Left: upstream cell $E^{n+1, n}_j$ and subregion $E^{n+1, n}_{j,l}$. Right: QC quadrilateral upstream cells.
}
\label{fig:schematic_2d_p2}
\end{figure}
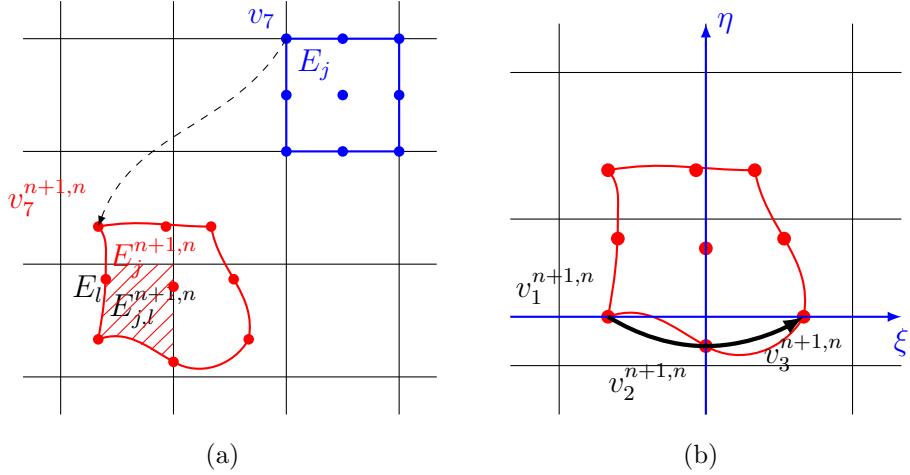

\noindent	
{\bf Step 2.2 Evaluation of Term II.}
Evaluation of Term II in eq.~\eqref{eq2:4:prop} can be realized by similar steps as in Step 1.2.
First, the LDG approximation to $u_{xx}$ in Step 1.2a can be directly generalized to evaluate $\Delta u$ in the 2D setting. $p = \Delta u$ is rewritten into a system of first order equations
\beq \label{eq2d_poisson}
	p = q_x+h_y, \quad
	q = u_x, \quad
	h = u_y.
\eeq
Weak formulations of \eqref{eq2d_poisson} can be discretized by an LDG method as in \eqref{eq2:12} to compute $p = \Delta u$ from five nearby elements. High order time discretizations of Term II can be fulfilled with DIRK methods on $\widetilde{E}_j(t)$ in the same fashion as Step 1.2b, by applying DIRK methods to the time differential form of the scheme \eqref{eq2:4:prop}
\beq \label{eq2:4:prop_diff}
\frac{d}{dt} \iint_{\widetilde{E}_j(t)}u \psi \ dx\ dy = \iint_{\widetilde{E}_j(t)} [\epsilon \Delta u+g] \psi \ dx \ dy.
\eeq
We summarize the flowchart below in {\bf Algorithm 1}, assuming the DIRK method has $s$-stages with the Butcher tableau
\begin{table}[htbp]
\centering
\begin{tabular}{c|c}
{\bf c}	&	A     \\
\hline
		& ${\bf b}^T$,
\end{tabular}
\end{table}\\
where $A = (a_{ ii , jj }) \in \mathbb{R}^{s \times s}$ , ${\bf b}\in \mathbb{R}^{s}$, and ${\bf c}\in \mathbb{R}^{s}$. In this paper, we assume that the DIRK method is stiffly accurate, i.e. the last row of the $A$ matrix is the same as the ${\bf b}^T$ vector.
Below we adopt a similar set of notations for the $E_j^{\tau_2, \tau}$ and $\psi^{\tau_2, \tau}(x, y)$ as in the 1D case, see \eqref{eq: 1d_notations}. In particular, $E_j^{\tau_2, \tau}$ is the upstream cell at time $\tau$, traced from the Eulerian cell $E_j$ at time $\tau_2$; and $\psi^{\tau_2, \tau}(x, y)$ is the approximation to the test function $\psi$ satisfying the adjoint problem \eqref{eq2:4:ad} with final value $ \psi(x, y, \tau_2) =\Psi(x, y) \in V^k_h$.

\bit

\item[(i)] In each DIRK stage $t^{( ii )}$, $1 \leq ii \leq s$, we have
\begin{equation*}
(u^{( ii )}, \Psi)_{E_j}	=	(u^n, \psi^{( ii ), n})_{E^{( ii ), n}_j}	+	\sum\limits_{ jj  = 1}^{ ii } a_{ ii , jj} \Dt  \left( \epsilon \Delta u^{( jj )}+g^{(jj)}, \psi^{( ii ), ( jj )} \right)_{E^{( ii ), ( jj )}_j}.
\end{equation*}
Rearranging the terms gives
\begin{multline} \label{eq:gen_int_u}
(u^{( ii )}, \Psi)_{E_j}	-	a_{ ii ,ii } \Dt \epsilon (\Delta u^{( ii )}, \Psi)_{E_j}		=
														(u^n, \psi^{(ii ), n})_{E^{(ii ), n}_j}	\\
														+	\sum\limits_{ jj = 1}^{ ii -1} a_{ ii , jj } \Dt \left( \epsilon \Delta u^{( jj )}+ g^{( jj )}, \psi^{(ii ), (jj)}\right)_{E^{( ii), ( jj )}_j}
														+	a_{ ii , ii } \Dt \left( g^{( ii )}, \Psi \right)_{E_j},
\end{multline}
which is a generalization from the scheme for 1D problems \eqref{eq3:5} and \eqref{eq3:11}.

\item[(ii)] Since the DIRK method we used is stiffly accurate, $u^{n+1}=u^{(s)}$.

\eit
\linespread{1.5}
\begin{algorithm} \small
\caption{The SLDG-LDG scheme coupled with DIRK methods}
	\begin{algorithmic}
	\For{$ ii = 1 \to s$}
			
		\For{$jj = 1 \to ii$}			
			
			Compute the RHS of eq. \eqref{eq:gen_int_u} by performing the following:
		
			\State 1. Find upstream cell $E^{(ii), (jj)}_j$ via tracing characteristic backwards by $(c_{ii}-c_{jj})\Dt$
			
			as in Step 2.1a.
						
			\State 2. Approximate test function $\psi^{(ii), (jj)}$ by a least-square strategy as in Step 2.1b.
						
			\State 3. Compute $(u^n, \psi^{(ii), n})_{E^{(ii), n}_j}$, $(\Delta u^{(jj)}, \psi^{(ii), (jj)})_{E^{(ii), (jj)}_j}$ and $\left(g^{(jj)}, \psi^{(ii), (jj)} \right)_{E^{(ii), (jj)}_j}$
			
			 by SLDG method as in Step 2.1c.
			
			\State 4. Evaluate $\left( g^{(ii)}, \Psi \right)_{E_j}$ by quadrature rules.
						
		\EndFor
		
		\State Compute $u^{(ii)}$ from eq.~\eqref{eq:gen_int_u}.
		
		\State Compute $\Delta u^{(ii)}$ with LDG method.
		
	\EndFor

	\State $u^{n+1} = u^{(s)}$.
	
	\end{algorithmic}
\label{algorithm}
\end{algorithm}

%------------------------------------------------------------------------------------------------------------------------------------------------------------------------------------------------------------------

Finally, we present the data structure for setting up the 2D SLDG-LDG implementation, which is similar to that in \cite{cai2017high}. There are two main classes as specified below. Please see Figure \ref{fig:class_diagram} in which the lines in the figure indicate connections between classes.
\bit
\item[(i)] \textbf{Cell-E} representing Eulerian cells, e.g. $E_j$ in Figure \ref{fig:schematic_2d_p2_a}. Main variables are
\bit
\item \textbf{Node-E}: $\{v_q\}^4_{q=1}$ for $P^0$ and $P^1$; $\{v_q\}^9_{q=1}$ for $P^2$ as vertices of Eulerian cell.
\item \textbf{SLDG-LDG solution}.
\eit
\item[(ii)] \textbf{Cell-U} representing upstream cells, e.g. $E^{n+1, n}_j$ in Figure \ref{fig:schematic_2d_p2_a}. Main variables are
\bit
\item \textbf{Node-U}: $\{v^{(ii),(jj)}_q\}^4_{q=1}$ for $P^0$ and $P^1$;  $\{v^{(ii),(jj)} _q\}^9_{q=1}$ for $P^2$ as vertices of upstream cell.
\item \textbf{Test function}: e.g. $\psi^{n+1,n}(x,y)$ approximated by a least-square procedure and by following characteristics of the adjoint problem, see Step 2.1b.
\eit
\eit

% Figure for data structure

\begin{figure}[h]
\centering
\begin{tikzpicture}[scale=0.9]

\tikzstyle{every node}=[font=\scriptsize]

\node (EulerianCell) [abstract, rectangle split, rectangle split parts=1]
at(-4.0,0.0)
	{
		\textbf{Cell-E}
	};
	
\node (SLDG-LDGsolution) [abstract, rectangle split, rectangle split parts=1]
at(-6.0,-1.5)
	{
		\textbf{SLDG-LDG solution}
	};
	
\node (NodeEulerian) [abstract, rectangle split, rectangle split parts=1]
at(-2.0,-1.5)
	{
		\textbf{Node-E}
	};
	
\node (UpstreamCell) [abstract, rectangle split, rectangle split parts=1]
at(4.0,0.0)
	{
		\textbf{Cell-U}
	};
	
\node (NodeUpstream) [abstract, rectangle split, rectangle split parts=1]
at(2.0,-1.5)
	{
		\textbf{Node-U}
	};

\node (TestfcnUp) [abstract, rectangle split, rectangle split parts=1]
at(6.0,-1.5)
	{
		\textbf{Test function}
	};

%-----------------------------------------

\draw[line](EulerianCell.south) -- ++(0,-0.3) -|(SLDG-LDGsolution.north);
\draw[line](EulerianCell.south) -- ++(0,-0.3) -|(NodeEulerian.north);

\draw[line](EulerianCell.east) -- (UpstreamCell.west);

\draw[line](UpstreamCell.south) -- ++(0,-0.3) -|(NodeUpstream.north);
\draw[line](UpstreamCell.south) -- ++(0,-0.3) -|(TestfcnUp.north);

\draw[line](NodeEulerian.east) -- (NodeUpstream.west);

\end{tikzpicture}

% Explanation of compositions of data structure

\linespread{1.8}

\begin{minipage}[t]{0.48\textwidth}\scriptsize
\bit
\item[(i)]
Cell-E: Eulerian Cell $E_j$. \\
Node-E: $\{v_q\}^4_{q=1}\ \text{for}\ P^0 \& P^1$ or $\{v_q\}^9_{q=1}\ \text{for}\ P^2$.\\
SLDG-LDG solution: $u^n \in  P^k(E_j)$.
\eit
\end{minipage}
\begin{minipage}[t]{0.48\textwidth}\scriptsize
\bit
\item[(ii)]
Cell-U: Upstream Cell $E^{( ii ), ( jj )}_j$.\\
Node-U: $\{v^{ (ii),(jj ) }_q\}^4_{q=1}\ \text{for}\ P^0 \& P^1$ or $\{v^{ (ii),(jj ) }_q\}^9_{q=1}\ \text{for}\ P^2$.\\
Test function: $\psi^{(ii), (jj)}$.\\
($ii, jj$ here refer to the indexes in the {\bf Algorithm \ref{algorithm}}).
\eit

\end{minipage}

\caption{Data structure of 2D SLDG-LDG schemes}

\label{fig:class_diagram}

\end{figure}

%---------------------------------------------------------------------------------------------------------------------------------------------------------------------------------------------------------------------

\subsection{Stability analysis}
\label{stab_anal}

We now briefly discuss the mass conservation and stability properties of the proposed SLDG-LDG schemes when coupled with first order backward Euler method.
Stability analysis of our scheme coupling with higher order time discretization will be pursued in the future.

% Mass conservative

\begin{prop} \label{prop_mass_con}
(Mass conservative). The SLDG-LDG method coupled with any DIRK time discretization methods for the linear convection-diffusion problems enjoy the mass conservation property, assuming the source term $g = 0$ in eq.~\eqref{eq3:1:1d} and periodic boundary condition.
$$
\int_{\Omega} u^{n+1} d \bx = \int_{\Omega} u^n d \bx.
$$
\end{prop}

\begin{proof}

This proposition can be easily proved by letting the test function $\Psi = 1$ in eq.~\eqref{eq3:3:2} and eq.~\eqref{eq2:4:prop} for 1D and 2D cases respectively, and then making use of the flux form of the LDG approximation of the diffusion term.

\end{proof}

% Stability proposition for backward Euler time discretization

\begin{prop}\label{prop_sta_be}
($L^2$ stability). Consider the proposed SLDG-LDG scheme coupled with the first-order backward Euler time discretization for 1D linear convection-diffusion equation $u_t + u_x = \epsilon u_{xx}$, $\epsilon>0$ and periodic boundary condition, then:
\beq \label{eq3:sta_be}
\Vert u^{n+1} \Vert  \leq \Vert u^n \Vert,
\eeq
where $\|\cdot\|$ denotes the standard $L^2$ norm over $\Omega$.
\end{prop}

\begin{proof}
The SLDG-LDG scheme for the 1D linear problem writes
\begin{subequations}
	\label{eq:stab_ldg}
	\begin{align}
	(u^{n+1},\Psi)_{I_j}-(u^n,\psi^{n+1, n})_{I^{n+1, n}_j}			&= \epsilon \cdot \Dt \left( ( \hat{q}^{n+1}\ \Psi ) \vert^{\jR}_{\jL} - (q^{n+1}, \Psi_x)_{I_j} \right),
	\label{stab_ldg1}	\\
	(q^{n+1},\varphi)_{I_j}					 		&= 	( \hat{u}^{n+1}\ \varphi ) \vert^{\jR}_{\jL} - (u^{n+1},\varphi_x)_{I_j} \label{stab_ldg2}
	\end{align}
\end{subequations}
where $\Psi, \varphi \in V^k_h$. As a standard technique for proving the stability,  we take the test function $\Psi= u^{n+1}$ and $\varphi = q^{n+1}$ on $I_j$ in eq.~\eqref{stab_ldg1} and eq.~\eqref{stab_ldg2}, respectively. According to \eqref{eq3:3}, we have
\[
\psi^{n+1, n}= u^{n+1}(x+\Dt)\doteq\delta_{\Dt}u^{n+1}.
\]
This, together with the weak formulations in eq.~\eqref{eq2:12}, yields
\begin{subequations}
\label{eq4:2}
	\begin{align}
	(u^{n+1},u^{n+1})_{I_j}-(u^n,\delta_{\Dt}u^{n+1})_{I^{n+1, n}_j}		
														&=	\epsilon \cdot \Dt \left( ( \hat{q}^{n+1}\ u^{n+1} ) \vert^{\jR}_{\jL} - (q^{n+1},u^{n+1}_x)_{I_j} \right),	\label{eq4:2a}\\
	(q^{n+1},q^{n+1})_{I_j} 									
														&=	( \hat{u}^{n+1}\ q^{n+1} ) \vert^{\jR}_{\jL} - (u^{n+1},q^{n+1}_x)_{I_j}.	\label{eq4:2b}
	\end{align}
\end{subequations}
$\eqref{eq4:2a}+\epsilon \cdot \Dt \cdot \eqref{eq4:2b}$ and summing up over $j$ give us

\begin{multline} \small
\label{eq4:3}
\Vert u^{n+1} \Vert ^2-\int_{x_a}^{x_b} u^n \cdot \delta_{\Dt}u^{n+1} dx+\epsilon \Dt \Vert q^{n+1}\Vert ^2 	=				\\
												\epsilon \cdot \Dt  \cdot \sum\limits_{j} \left\{ ( \hat{q}^{n+1}\ u^{n+1} ) \vert^{\jR}_{\jL} + ( \hat{u}^{n+1}\ q^{n+1} ) \vert^{\jR}_{\jL}
											-	\left[(q^{n+1},u^{n+1}_x)_{I_j}+(u^{n+1},q^{n+1}_x)_{I_j} \right]	\right\}		\\
											=	\epsilon \cdot \Dt \cdot \sum\limits_{j} \left\{ ( \hat{q}^{n+1}\ u^{n+1} ) \vert^{\jR}_{\jL} + ( \hat{u}^{n+1}\ q^{n+1} ) \vert^{\jR}_{\jL}
											-	\int_{I_j} (q^{n+1} u^{n+1})_x\ dx	  \right\}		\\
											=	\epsilon \cdot \Dt \cdot \sum\limits_{j} \left\{ ( \hat{q}^{n+1}\ u^{n+1} ) \vert^{\jR}_{\jL} + ( \hat{u}^{n+1}\ q^{n+1} ) \vert^{\jR}_{\jL}
											-	( q^{n+1} u^{n+1} ) \vert^{\jR}_{\jL}	\right\}	=	0,
\end{multline}where  the cancellation is due to the alternating fluxes used, see \eqref{eq2:13}, and the periodicity.

Noting that $\epsilon \Dt \Vert q^{n+1}\Vert ^2\geq 0$ on the LHS of \eqref{eq4:3}, we have
\beq \label{eq4:4}
\Vert u^{n+1} \Vert ^2 \leq \int_{x_a}^{x_b} u^n \cdot \delta_{\Dt}u^{n+1} dx.
\eeq
Then, applying the Cauchy-Schwarz inequality to \eqref{eq4:4} together with the identity $\|\delta_{\Dt}u^{n+1}\|=\|u^{n+1}\|$ yields
\beq \label{eq4:5}
\Vert u^{n+1} \Vert  \leq \Vert u^n \Vert.
\eeq
This completes the proof.
\end{proof}

\section{Numerical tests}
\label{sec5}

\setcounter{equation}{0}

In this section, we present the convergence study in terms of spatial and temporal orders of the proposed SLDG-LDG methods for a collection of 1D and 2D benchmark linear convection-diffusion equations. Mass conservation in Proposition~\ref{prop_mass_con} is also numerically verified. We assume uniform partition of the computational domain $\Omega$. In principle, the addressed schemes can be extended to general nonuniform meshes.
We let $\Dt = \f{\text{CFL}}{\f{\max\vert a(x,t)\vert}{\Dx}}$ and $\Dt = \f{\text{CFL}}{\f{\max\vert a(x,y,t)\vert}{\Dx}+\f{\max\vert b(x,y,t)\vert}{\Dy}}$ for 1D and 2D tests respectively, in which the CFL number is to be specified.
For the test of spatial accuracy, we choose DIRK4 with Butcher tableau specified in Table \ref{tab_dirk4} in order minimize the time discretization error.  Likewise, for temporal accuracy tests, we use SLDG-LDG with piecewise $P^2$ polynomial space unless otherwise specified.  

%---------------------------------------------------------------------------------------------------------------------------------------------------------------------------
% 1D tests

\subsection{One-dimensional tests}

%----------------------------------

\begin{exa} \label{exa_1d_const}

( 1D linear convection-diffusion equation.)  
Consider the following 1D convection-diffusion equation
\beq \label{eq_exa1}
u_t + u_x 	=	 \epsilon u_{xx}, \quad x\in[0, 2\pi]
\eeq
with exact solution $u = \sin(x-t)\exp(-\epsilon t)$. Table \ref{spatial_1d_const} provides $L^1$, $L^2$, $L^{\infty}$ errors to verify the spatial performance of the SLDG-LDG method. One can observe the method is of $k+1$ order when $V^k_h$  is used, as expected. To demonstrate the temporal orders of accuracy of the employed DIRK methods, we fix $N = 500$ and let the CFL numbers vary from $1.1$ to $12.1$ (note the extra large values of CFL numbers). Figure \ref{fig:1D_con_temp} shows the $L^1$ error of the proposed scheme coupled with DIRK2, DIRK3 and DIRK4 and expected orders can be observed as compared with reference slopes. Note that for this problem there is no error incurred in time for the convection part, since the characteristics are tracked exactly in the proposed SL setting.

\linespread{1.0}

\begin{table} [htbp]\scriptsize
\centering
\caption{Spatial order of accuracy for Example \ref{exa_1d_const} with $CFL = 1.0, \epsilon = 1$ at $T = 1.0$.}
\label{spatial_1d_const}

\bigskip

\begin{tabular}{|c | cc |cc | cc |}
\hline
\multicolumn{7}{|c|}{$k = 0$}\\
\hline
 mesh &{$L^1$ error} & Order    &{$L^2$ error} & Order &{$L^{\infty}$ error} & Order\\
\hline
    10 &     3.79E-02 &		    &     4.78E-02 &		&     1.08E-01 & 		   \\
    20 &     1.92E-02 &     0.98 &     2.40E-02 &     0.99 	&     5.45E-02 &     0.99 \\
    40 &     9.41E-03 &     1.03 &     1.18E-02 &     1.02 	&     2.70E-02 &     1.02 \\
    80 &     4.70E-03 &     1.00 &     5.90E-03 &     1.00 	&     1.35E-02 &     1.00 \\
   160 &     2.35E-03 &     1.00 &     2.95E-03 &     1.00 &     6.74E-03 &     1.00 \\
\hline
\multicolumn{7}{|c|}{$k = 1$}\\
\hline
 mesh &{$L^1$ error} & Order    &{$L^2$ error} & Order &{$L^{\infty}$ error} & Order\\
\hline
    10 &     4.60E-03 &		    &     5.57E-03 &		&     1.15E-02 & 		   \\
    20 &     1.21E-03 &     1.92 &     1.50E-03 &     1.90 	&     4.27E-03 &     1.43 \\
    40 &     2.88E-04 &     2.07 &     3.70E-04 &     2.01 	&     1.17E-03 &     1.87 \\
    80 &     7.01E-05 &     2.04 &     9.28E-05 &     2.00 	&     3.04E-04 &     1.94 \\
   160 &    1.78E-05 &     1.98 &     2.39E-05 &     1.96  &     7.95E-05 &     1.94 \\
\hline
\multicolumn{7}{|c|}{$k = 2$}\\
\hline
 mesh &{$L^1$ error} & Order    &{$L^2$ error} & Order &{$L^{\infty}$ error} & Order\\
\hline
    10 &     2.18E-04 &		    &     3.19E-04 &		&     1.08E-03 & 		   \\
    20 &     2.57E-04 &     3.09 &     3.92E-05 &     3.03 	&     1.36E-04 &     2.99 \\
    40 &     3.32E-06 &     2.95 &     5.05E-06 &     2.96 	&     1.77E-05 &     2.95 \\
    80 &     4.00E-07 &     3.05 &     6.02E-07 &     3.07 	&     2.05E-06 &     3.11 \\
   160 &    5.10E-08 &     2.97 &     7.73E-08 &     2.96  &     2.68E-07 &     2.94 \\
\hline
\end{tabular}

\end{table}

%--------------------------------------
% Temporal order of accuracy (figure)

\begin{figure}
\centering
	\includegraphics[width = 0.5\textwidth]{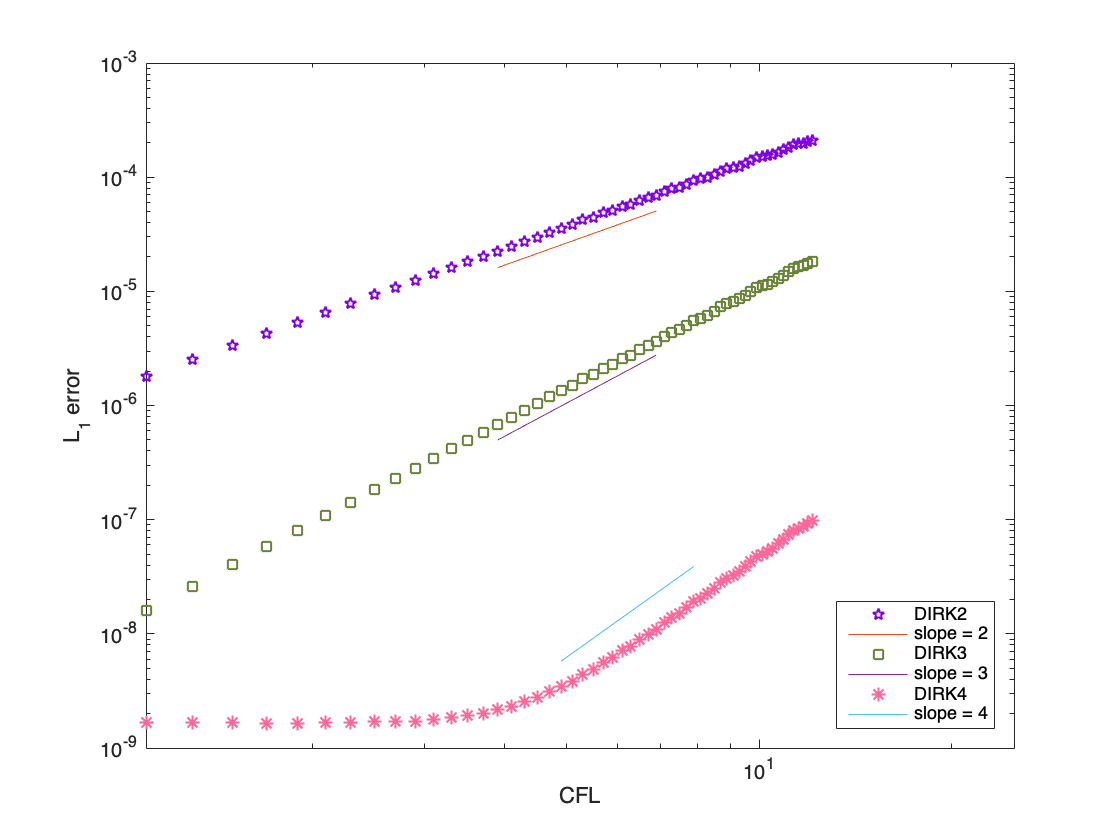}
\caption{ $L^1$ error with varying CFL numbers of Example \ref{exa_1d_const}. SLDG-LDG with $P^2$ polynomial space. $N=500$ for spatial discretization.}
\label{fig:1D_con_temp}
\end{figure}

\end{exa}

%-------------------------------------------------------------------------------------

\begin{exa} \label{exa_1d_var}

( 1D equation with variable coefficient.)  
Consider
\beq \label{eq_exa2.1}
u_t + (\sin(x)u)_x 		=	 \epsilon u_{xx}+g, \quad x\in[0, 2\pi]
\eeq
with $u = \sin(x)\exp(-\epsilon t)$ and $g = \sin(2 x)\exp(-\epsilon t)$. Expected spatial orders of accuracy are observed in Table \ref{spatial_1d_var}. With fixed mesh $N = 500$ and CFL varying from $0.3$ to $12.1$, high order temporal convergences are observed in Figure \ref{fig:1D_var_temp}. 

%------------------------------------------------
% Spatial order of accuracy (table)

\linespread{1.0}

\begin{table} [htbp]\scriptsize
\centering
\caption{Spatial order of accuracy for Example \ref{exa_1d_var} with $CFL = 1.0, \epsilon = 1$ at $T = 1.0$.}
\label{spatial_1d_var}

\bigskip

\begin{tabular}{|c | cc |cc | cc |}
\hline
\multicolumn{7}{|c|}{$k = 0$}\\
\hline
 mesh &{$L^1$ error} & Order    &{$L^2$ error} & Order &{$L^{\infty}$ error} & Order\\
\hline
    10 &     4.20E-02 &		    &     4.96E-02 &		&     1.11E-01 & 		   \\
    20 &     1.97E-02 &     1.09 &     2.42E-02 &     1.03 	&     5.41E-02 &     1.05 \\
    40 &     9.96E-03 &     0.98 &     1.22E-02 &     0.99 	&     2.71E-02 &     1.00 \\
    80 &     4.97E-03 &     1.00 &     6.11E-03 &     0.99 	&     1.35E-02 &     1.00 \\
   160 &     2.50E-03 &     1.00 &     3.07E-03 &     1.00 &     6.80E-03 &     0.99 \\
\hline
\multicolumn{7}{|c|}{$k = 1$}\\
\hline
 mesh &{$L^1$ error} & Order    &{$L^2$ error} & Order &{$L^{\infty}$ error} & Order\\
\hline
    10 &     6.24E-03 &		    &     8.42E-03 &		&     3.11E-02 & 		   \\
    20 &     1.33E-03 &     2.23 &     1.78E-03 &     2.25 	&     6.53E-03 &     2.25 \\
    40 &     3.06E-04 &     2.12 &     3.20E-04 &     2.08 	&     1.57E-03 &     2.05 \\
    80 &     7.39E-05 &     2.05 &     1.04E-04 &     2.02 	&     3.91E-04 &     2.01 \\
   160 &    1.85E-05 &     2.00 &     2.62E-05 &     1.98  &     9.42E-05 &     2.05 \\
\hline
\multicolumn{7}{|c|}{$k = 2$}\\
\hline
 mesh &{$L^1$ error} & Order    &{$L^2$ error} & Order &{$L^{\infty}$ error} & Order\\
\hline
    10 &     4.29E-04 &		    &     5.38E-04 &		&     1.69E-03 & 		   \\
    20 &     9.53E-05 &     2.17 &     1.09E-04 &     2.31 	&     2.60E-04 &     2.70 \\
    40 &     8.16E-06 &     3.55 &     9.63E-06 &     3.49 	&     2.70E-05 &     3.27 \\
    80 &     7.72E-07 &     3.40 &     9.37E-07 &     3.36 	&     3.03E-06 &     3.15 \\
   160 &    7.57E-08 &     3.35 &     9.60E-08 &     3.29  &     3.39E-07 &     3.16 \\
\hline
\end{tabular}

\end{table}

%---------------------------------------------------------------------------------------------------------
% Temporal order of accuracy (figure)

\begin{figure}
\centering
	\includegraphics[width = 0.5\textwidth]{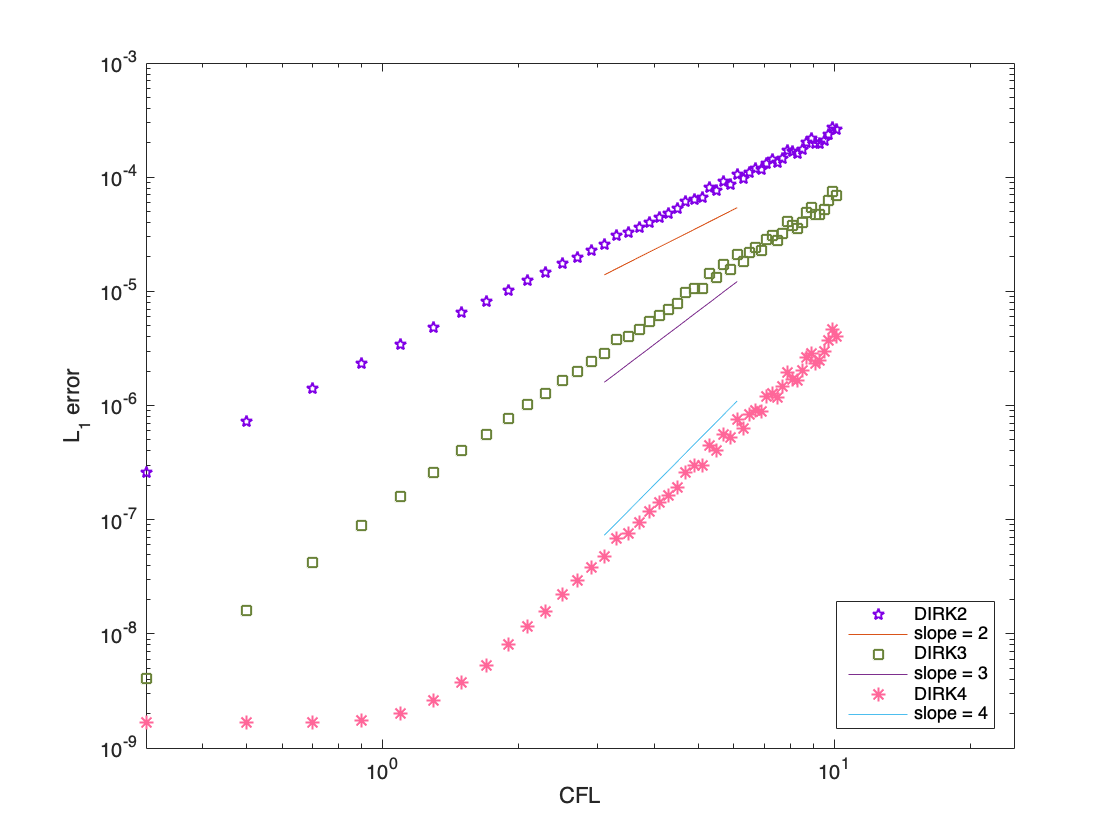}
\caption{$L^1$ error with varying CFL numbers of Example \ref{exa_1d_var}. SLDG-LDG with $P^2$ polynomial space. $N=500$ for spatial discretization.
}
\label{fig:1D_var_temp}
\end{figure}

\end{exa}

%-------------------------------------------------------------------------------------------------------------------------------------------------------------------------------------------
% 2D tests

\subsection{Two-dimensional tests}

%---------------------------------------------------------------------------------------------------------

\begin{exa}\label{exa_2d_const}

(2D linear convection-diffusion equation.) Consider
\beq
u_t + u_x+u_y 	=	 \epsilon \Delta u, \quad x, y \in[0, 2\pi]
\label{eq_exa3}
\eeq
with exact solution $u = \sin(x+y-2t)\exp(-2\epsilon t)$.

We observe $k+1$ spatial orders of accuracy from Table \ref{spatial_2d_const} when $V_h^k$ is employed. The temporal convergence study is summarized in Figure \ref{fig:2D_con_temp} with fixed a mesh of $J = 200^2$ cells. The observation is similar to that in the 1D case (Example 3.1). 
To verify the mass conservation property claimed in Proposition 2.6 for Example 3.3, we plot $\vert \iint_{\Omega} u^n d{\bf x} \vert$ with $V^0_h$ computed using different thresholds $10^{-10}$, $10^{-12}$ and $10^{-14}$ for the GMRES iterative method in Figure 9. We see that the variation of $\vert \iint_{\Omega} u^n d {\bf x} \vert$ in time depends on the choice of threshold and it gets closer to the machine precision when a smaller threshold in GMRES is used.

%----------------------------------
% Spatial order of accuracy (table)

\linespread{1.0}

\begin{table} [htbp]\scriptsize
\centering
\caption{Spatial order of accuracy for Example \ref{exa_2d_const} with $CFL = 1.0, \epsilon = 1$ at $T = 1.0$.}
\label{spatial_2d_const}
\bigskip

\begin{tabular}{|c | cc |cc | cc |}
\hline
\multicolumn{7}{|c|}{$k = 0$; Quadrilateral}\\
\hline
 mesh &{$L^1$ error} & Order    &{$L^2$ error} & Order &{$L^{\infty}$ error} & Order\\
\hline
$20^2$	&     4.64E-02 	& 		&     5.15E-02 & 		&     7.67E-02 & 			\\
$60^2$	&     1.99E-02 	&     0.77 &     2.21E-02 &     0.77 &     3.30E-02 &     0.77 \\
$100^2$ &     1.26E-02 	&     0.89 &     1.40E-02 &     0.89 &     2.09E-02 &     0.89 \\
$140^2$ &     9.23E-03 	&     0.93 &     1.03E-02 &     0.93 &     1.53E-02 &     0.93 \\
$180^2$ &     7.27E-03 	&     0.95 &     8.08E-03 &     0.95 &     1.20E-02 &     0.95 \\
\hline
\multicolumn{7}{|c|}{$k = 1$; Quadrilateral}\\
\hline
 mesh &{$L^1$ error} & Order    &{$L^2$ error} & Order &{$L^{\infty}$ error}& Order\\
\hline
$20^2$ 	&     1.10E-03 & 		&     1.35E-03 & 		&     5.23E-03 & 			\\
$60^2$ 	&     9.59E-05 &     2.22 &     1.28E-04 &     2.14 &     6.94E-04 &     1.84 \\
$100^2$ &     3.28E-05 &     2.10 &     4.52E-05 &     2.04 &     2.57E-04 &     1.95 \\
$140^2$	&     1.65E-05 &     2.05 &     2.31E-05 &     2.00 &     1.33E-04 &     1.95 \\
$180^2$ &     9.87E-06 &     2.04 &     1.40E-05 &     2.00 &     8.14E-05 &     1.96 \\
\hline
\multicolumn{7}{|c|}{ $k = 2$; Quadrilateral}\\
\hline
 mesh &{$L^1$ error} & Order    &{$L^2$ error} & Order &{$L^{\infty}$ error}& Order\\
\hline
$20^2$ 	&     4.14E-05 & 		&     6.06E-05 & 		&     4.82E-04 & 			\\
$60^2$  	&     1.59E-06 &     2.96 &     2.35E-06 &     2.96 &     1.88E-05 &     2.95 \\
$100^2$ &     3.45E-07 &     3.00 &     5.09E-07 &     2.99 &     4.08E-06 &     2.99 \\
$140^2$	&     1.26E-07 &     2.99 &     1.86E-07 &     2.99 &     1.49E-06 &     3.00 \\
$180^2$	&     5.96E-08 &     2.98 &     8.78E-08 &     2.99 &     7.03E-07 &     2.99 \\
\hline
\end{tabular}

\end{table}

%----------------------------------
% Temporal order of accuracy (figure)

\begin{figure}
\centering
	\includegraphics[width = 0.5\textwidth]{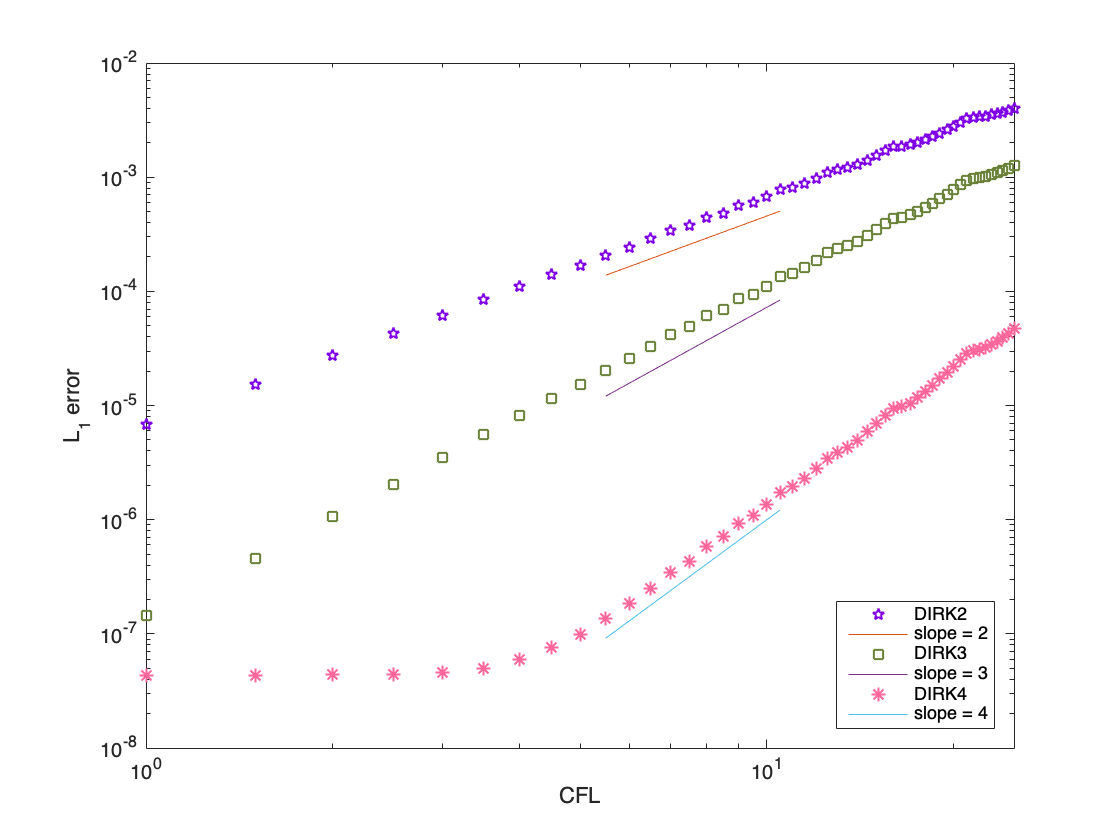}
\caption{$L_1$ error with varying CFL number of Example~\ref{exa_2d_const}. From top to bottom: DIRK2, DIRK3 and DRIK4.}
\label{fig:2D_con_temp}
\end{figure}

% Conservation test
\begin{figure}[htbp]
\centering
    \includegraphics[width=0.6\textwidth]{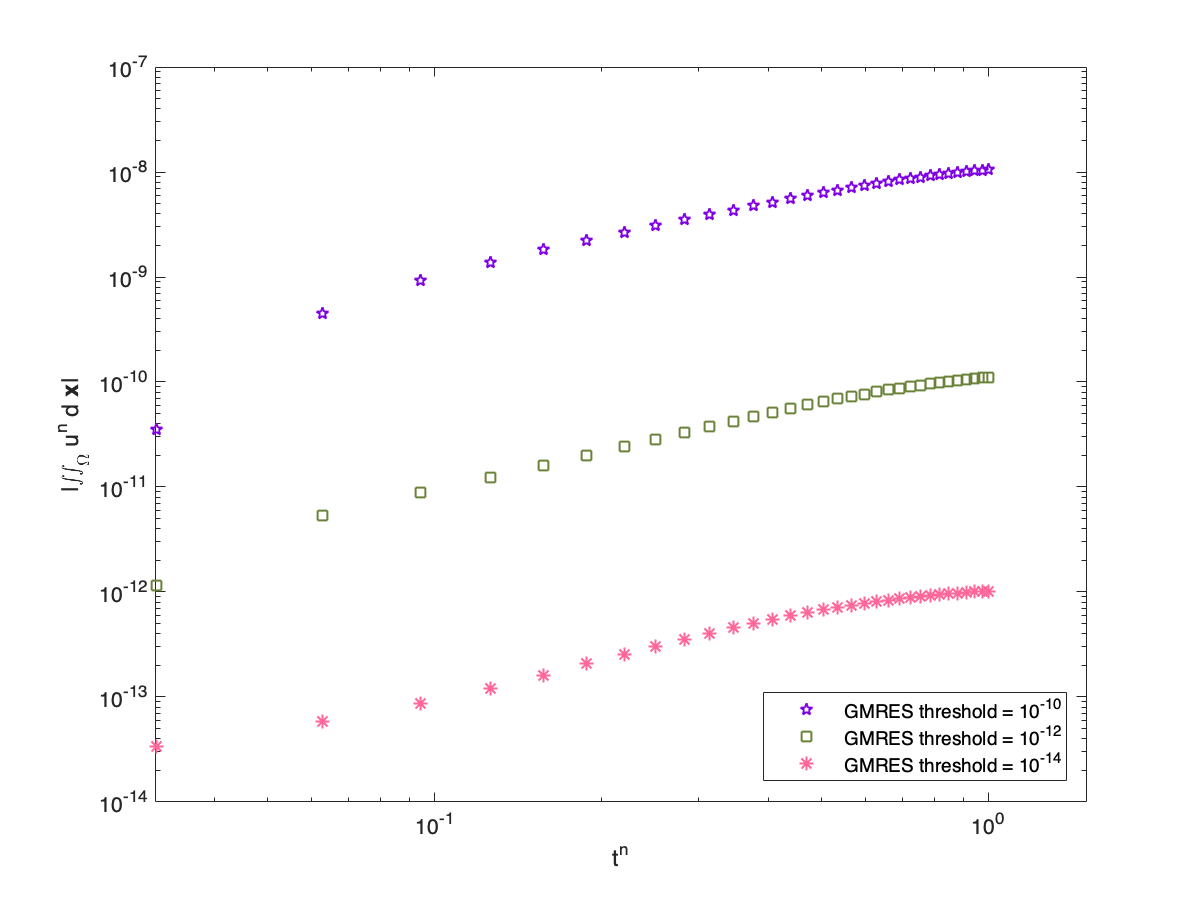}
\caption{$\vert \iint_{\Omega} u^n \ d\bf{x} \vert$ at different $t^n$ for Example~\ref{exa_2d_const} with $V^0_h$ space. From top to bottom: GMRES iterative method with thresholds $10^{-10}$, $10^{-12}$ and $10^{-14}$.}
\label{fig:2D_con_conser}
\end{figure}

\end{exa}

%---------------------------------------------------------------------------------------------------------

\begin{exa} \label{exa_2d_rigid}

(Rigid body rotation)  Consider
\beq \label{eq_exa4.1}
u_t - (yu)_x+(xu)_y	 =	 \epsilon \Delta u+g, \quad x , y \in[-2\pi, 2\pi].
\eeq
In order to test the spatial order of accuracy, we choose $u = \exp[ -(x^2+3y^2+2\epsilon t) ]$ and $g = [6\epsilon-4xy-4\epsilon\left(x^2+9y^2\right)] \exp[ -(x^2+3y^2+2\epsilon t) ]$. The results are summarized in Table \ref{spatial_2d_rigid} for $\epsilon = 1.0, CFL = 10.0$. The choice of the large CFL number once again supports our claim that the proposed methods are free of the stringent CFL restriction, leading to computational savings. 
Figure~\ref{fig:2D_rigid_temp} displays the temporal convergence study where we use two different ways to compute the numerical error: test1 is done by using the exact solution on mesh $J=200^2$; test2 uses the reference solution when $CFL = 0.01$ on mesh $J = 150^2$ and the final time $T = 0.1$. In Figure~\ref{fig:2D_rigid_temp}, when test1 is used, expected orders for DIRK2 and DIRK3 are observed, while the numerical convergence order is not clear for DIRK4. In order to reduce the interference of the spatial error, we carry out test2 and observe the expected order of temporal accuracy.

We numerically solve eq.~\eqref{eq_exa4.1} with an initial condition plotted in Figure \ref{fig:disk_cone_hump} which includes a slotted disk, a cone as well as a smooth hump. 
$\epsilon = 0.01,\, g = 0$ and mesh size $J = 200^2$ are chosen here.
The numerical solutions and the corresponding contour plots obtained by the proposed SLDG-LDG schemes with $V^k_h$ space ($k = 1, 2$) after $T = 1.0$ are plotted in Figure \ref{fig:rigid_disk_cone_hump}. DIRK4 time discretization method is applied.

To better compare the performances of the schemes with $V^1_h$ and $V^2_h$ spaces for the DG and LDG discretization, we plot the 1D cuts of the numerical solutions at $X = -1.0$ and $Y = 1.0$ with mesh size $J = 50^2$ along with the one computed on a refined mesh with $J = 200^2$ as in Figure \ref{fig:1D_cuts_rigid} at $T = 1.0$. 

%----------------------------------
% Spatial order of accuracy (table)

\linespread{1.0}

\begin{table} [htbp]\scriptsize
\centering
\caption{Spatial order of accuracy for Example \ref{exa_2d_rigid} with CFL = 10.0, $\epsilon = 1.0$ at $T = 1.0$.}
\label{spatial_2d_rigid}
\bigskip

\begin{tabular}{|c | cc |cc | cc |}
\hline
\multicolumn{7}{|c|}{$k = 0$; Quadrilateral}\\
\hline
 mesh &{$L^1$ error} & Order    &{$L^2$ error} & Order &{$L^{\infty}$ error} & Order\\
\hline
$20^2$  &     1.97E-03 &		 &     6.53E-03 & 		&     9.39E-02 & 			\\
$60^2$  &     8.86E-04 &     0.73 &     3.08E-03 &     0.68 &     4.62E-02 &     0.65 \\
$100^2$ &     5.73E-04 &     0.85 &     2.02E-03 &     0.82 &     2.99E-02 &     0.85 \\
$140^2$ &     4.23E-04 &     0.90 &     1.51E-03 &     0.88 &     2.23E-02 &     0.87 \\
$180^2$ &     3.36E-04 &     0.92 &     1.20E-03 &     0.91 &     1.78E-02 &     0.90 \\
\hline
\multicolumn{7}{|c|}{$k = 1$; Quadrilateral}\\
\hline
 mesh &{$L^1$ error} & Order    &{$L^2$ error} & Order &{$L^{\infty}$ error}& Order\\
\hline
$20^2$ 	&     2.76E-04 & 		&     1.30E-03 & 		&     2.81E-02 &			 \\
$60^2$	 &     2.68E-05 &     2.12 &     1.58E-04 &     1.92 &     5.75E-03 &     1.45 \\
$100^2$	&     9.33E-06 &     2.07 &     5.72E-05 &     1.99 &     2.25E-03 &     1.83 \\
$140^2$  &     4.70E-06 &     2.04 &     2.94E-05 &     1.98 &     1.19E-03 &     1.91 \\
$180^2$	 &     2.82E-06 &     2.03 &     1.78E-05 &     1.98 &     7.29E-04 &     1.94 \\
\hline
\multicolumn{7}{|c|}{$k = 2$; Quadrilateral}\\
\hline
 mesh &{$L^1$ error} & Order    &{$L^2$ error} & Order &{$L^{\infty}$ error}& Order\\
\hline
$20^2$ 	&     7.11E-05 & 		&     3.51E-04 & 		&     1.32E-02 & 			\\
$60^2$ 	&     1.92E-06 &     3.29 &     1.18E-05 &     3.08 &     5.31E-04 &     2.92 \\
$100^2$ &     4.04E-07 &     3.05 &     2.57E-06 &     2.99 &     1.14E-04 &     3.01 \\
$140^2$ &     1.46E-07 &     3.03 &     9.39E-07 &     3.00 &     4.24E-05 &     2.94 \\
$180^2$ &     6.82E-08 &     3.02 &     4.41E-07 &     3.00 &     2.00E-05 &     2.99 \\
\hline
\end{tabular}

\end{table}

%----------------------------------
% Temporal order of accuracy (figure)

\begin{figure}
\begin{minipage}[b]{0.48\textwidth}
\centering
	\includegraphics[width = 1.0\textwidth]{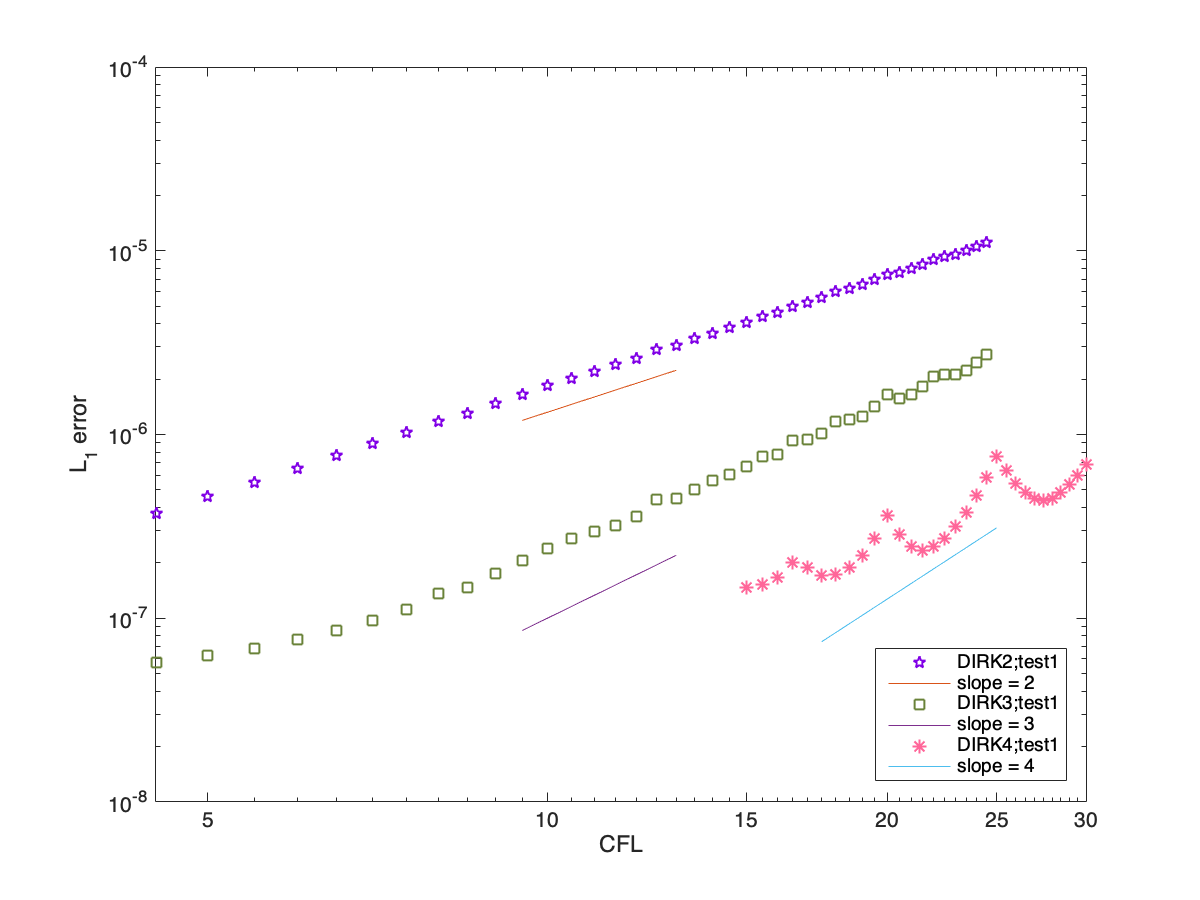}
\end{minipage}
\hfill
\begin{minipage}[b]{0.48\textwidth}
\centering
	\includegraphics[width = 1.0\textwidth]{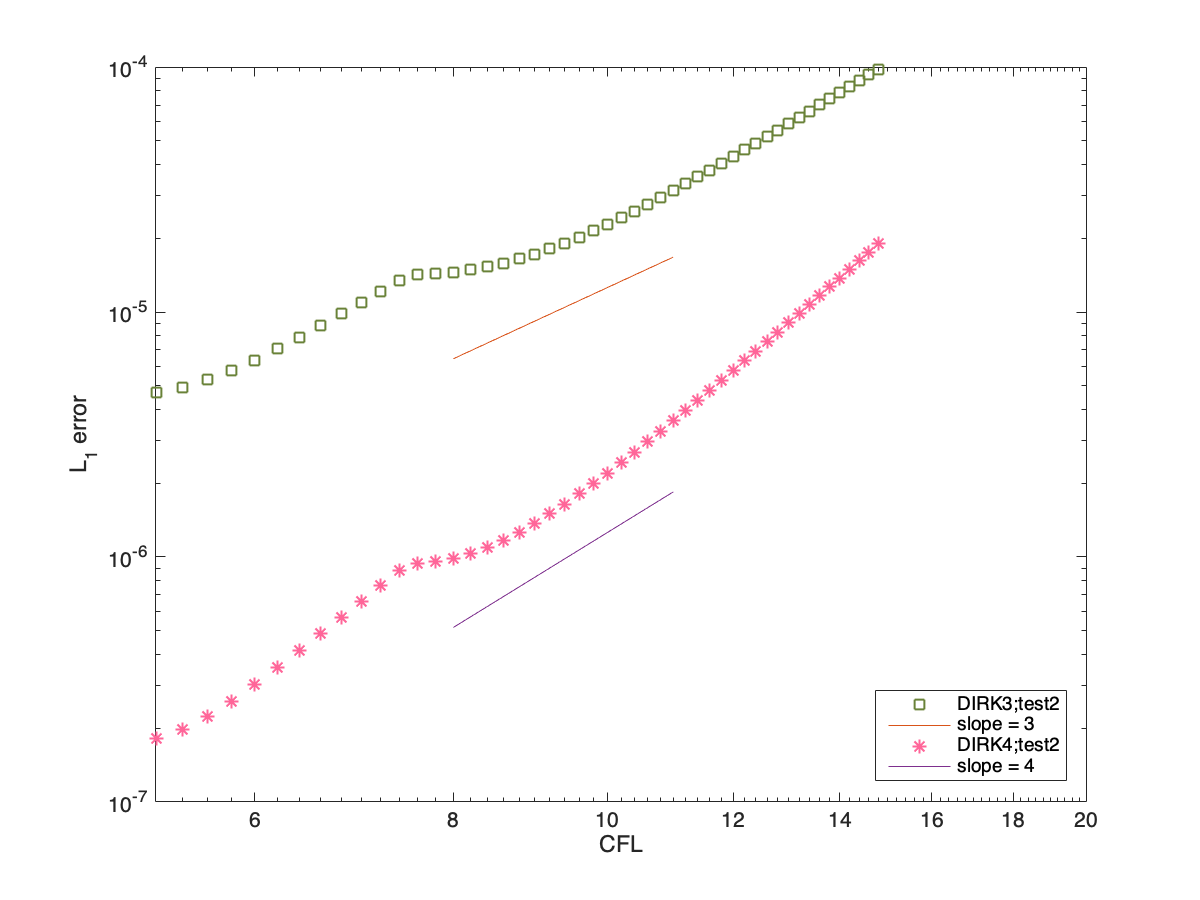}
\end{minipage}
\caption{$L^1$ error with varying CFL number of Example~\ref{exa_2d_rigid}. Left: test1 using exact solution on mesh $J = 200^2$. Right: test2 using reference solution when $CFL = 0.01$ on mesh $J = 150^2$.}
\label{fig:2D_rigid_temp}
\end{figure}

\end{exa}

%--------------------------------------------

\begin{figure}
\begin{minipage}[b]{0.48\textwidth}
\centering
    \includegraphics[width=\textwidth]{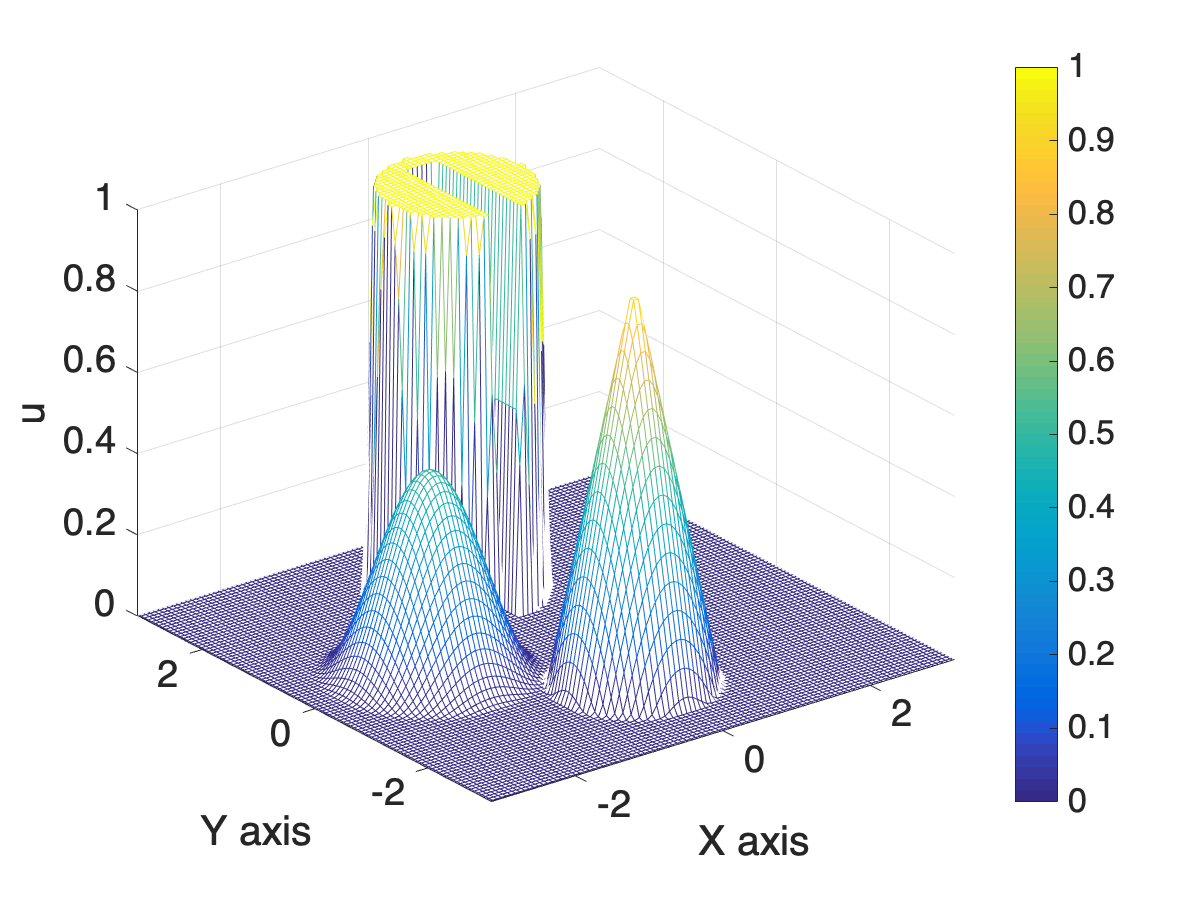}
\end{minipage}
\hfill
\begin{minipage}[b]{0.48\textwidth}
\centering
    \includegraphics[width=\textwidth]{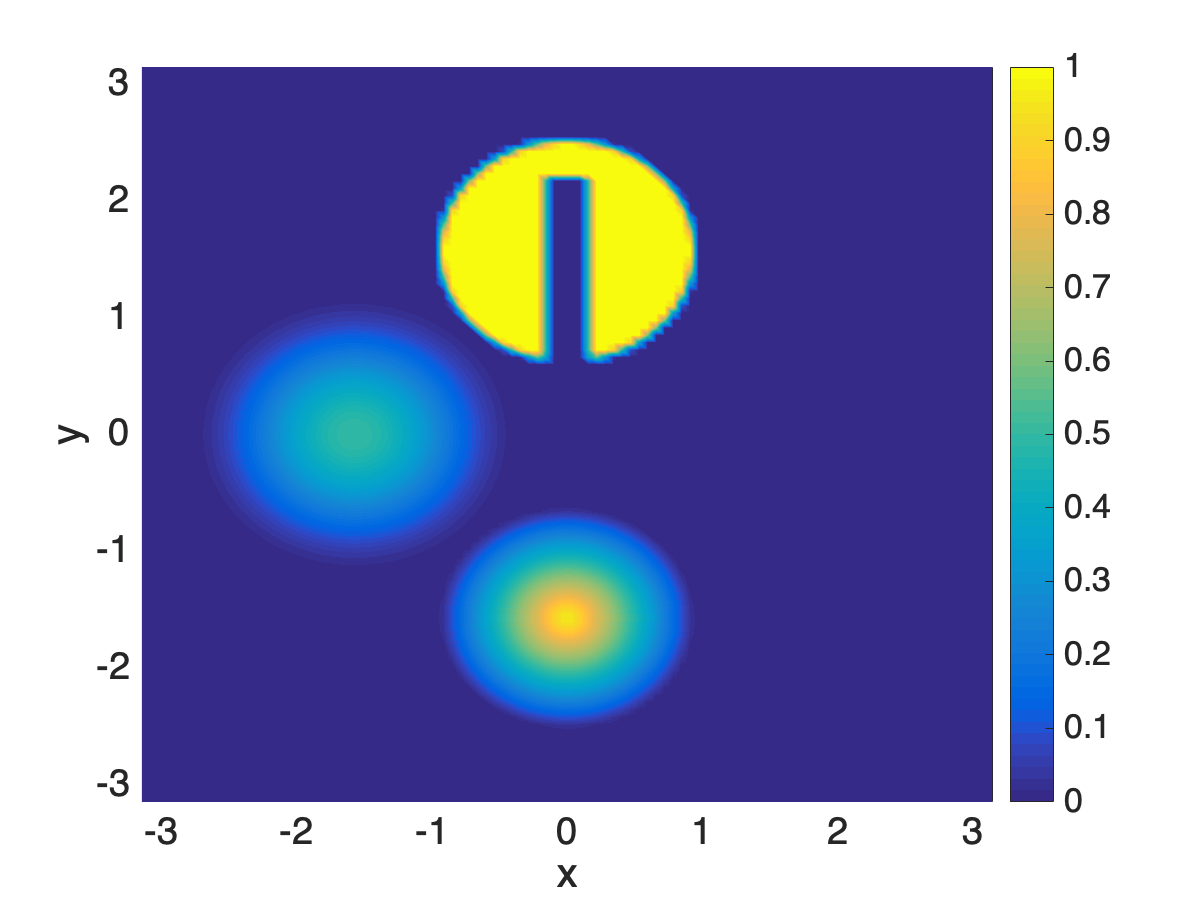}
\end{minipage}
\caption{Plots of the initial profile consisting of a slotted disk, a cone and a smooth hump for Example 3.4. Mesh size is $200 \times 200$. Left: Initial condition. Right: Contour plot.}
\label{fig:disk_cone_hump}
\end{figure}

%----------------------------------

\begin{figure}
\begin{minipage}[b]{0.45\textwidth}
\centering
    \includegraphics[width=\textwidth]{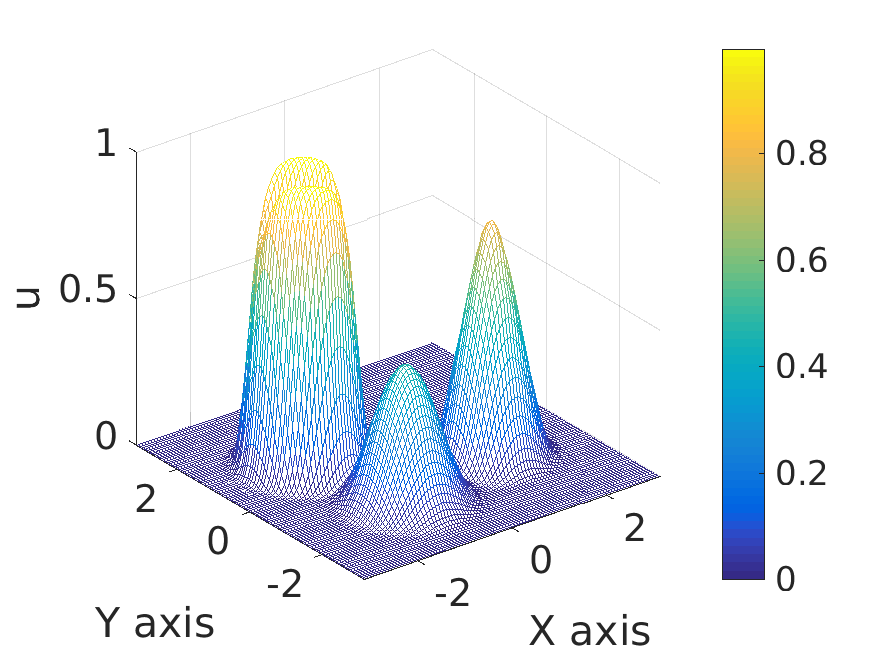}
\end{minipage}
\hfill
\begin{minipage}[b]{0.45\textwidth}
\centering
    \includegraphics[width=\textwidth]{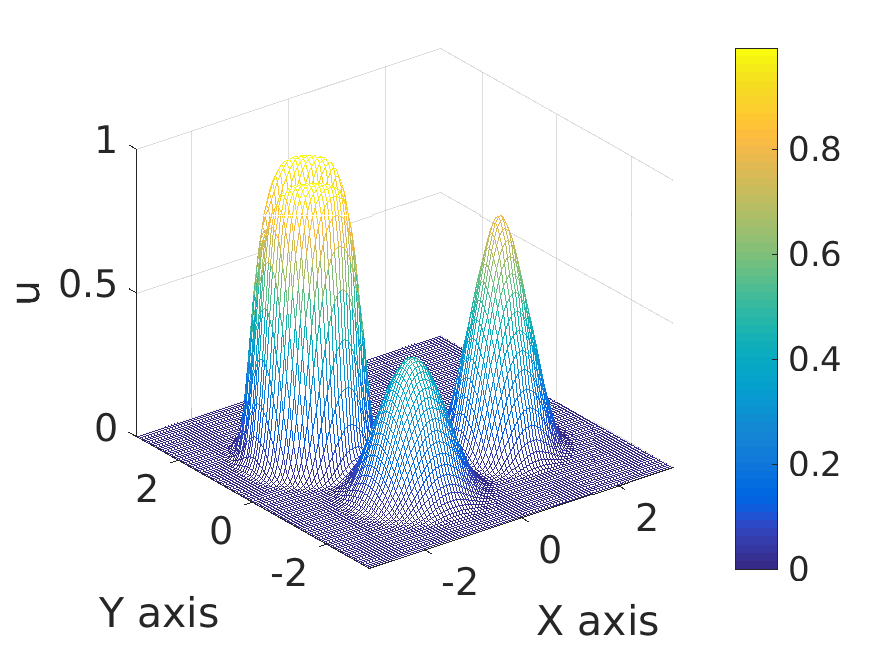}
\end{minipage}
\begin{minipage}[b]{0.45\textwidth}
\centering
    \includegraphics[width=\textwidth]{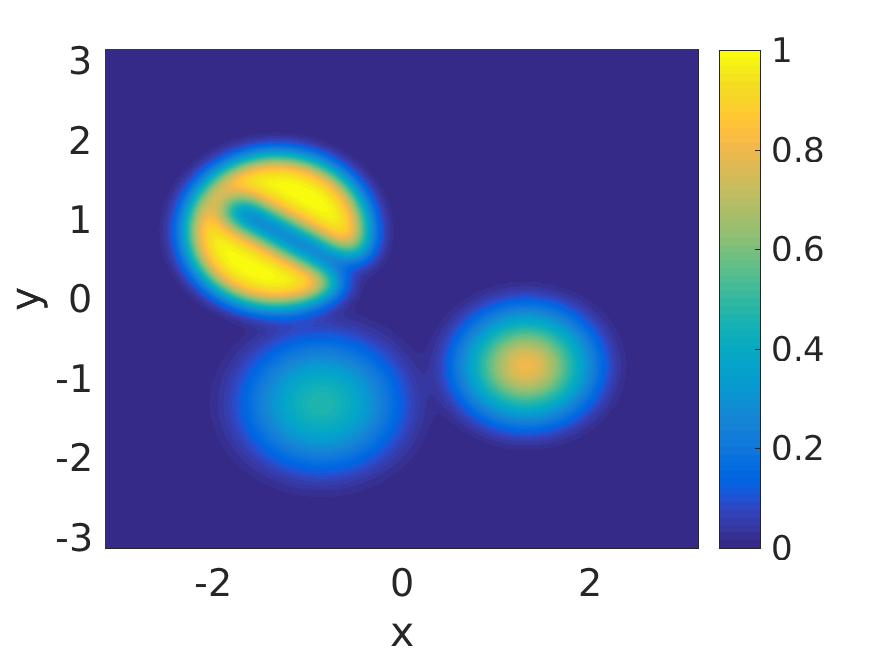}
\end{minipage}
\hfill
\begin{minipage}[b]{0.45\textwidth}
\centering
    \includegraphics[width=\textwidth]{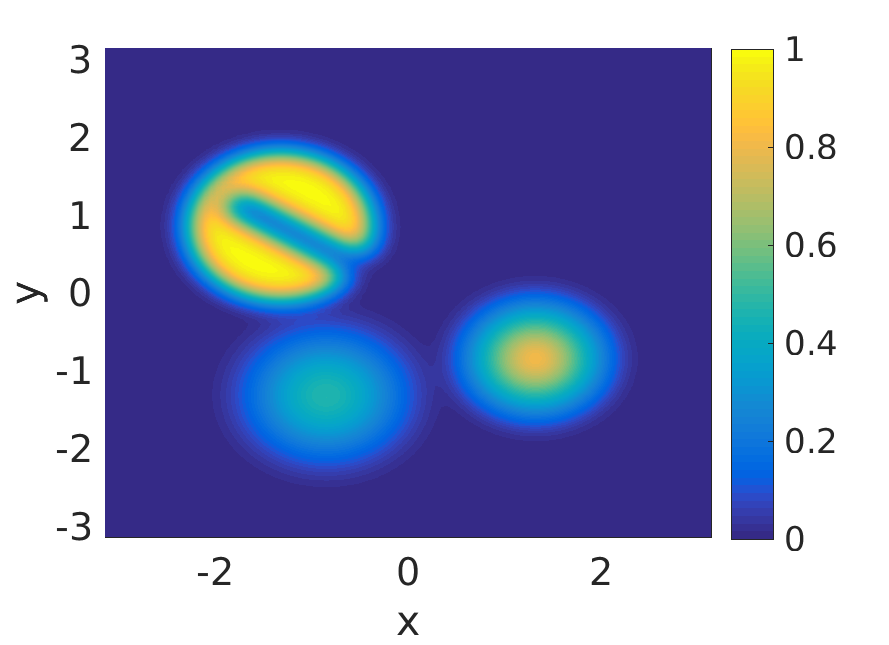}
\end{minipage}
\caption{\small{First row consists of plots for numerical solutions with SLDG-LDG for equation \eqref{eq_exa4.1} with initial data Figure \ref{fig:disk_cone_hump} and the second row consists of the corresponding contour plots. Mesh size is $200 \times 200$. Final integration time $T = 1.0$. $\Dt = 2.5\Dx$.
From left to right: $P^1$ SLDG-LDG, $P^2$ SLDG-LDG. }}
\label{fig:rigid_disk_cone_hump}
\end{figure}

%----------------------------------

\begin{figure}
\begin{minipage}[b]{0.45\textwidth}
\centering
    \includegraphics[width=\textwidth]{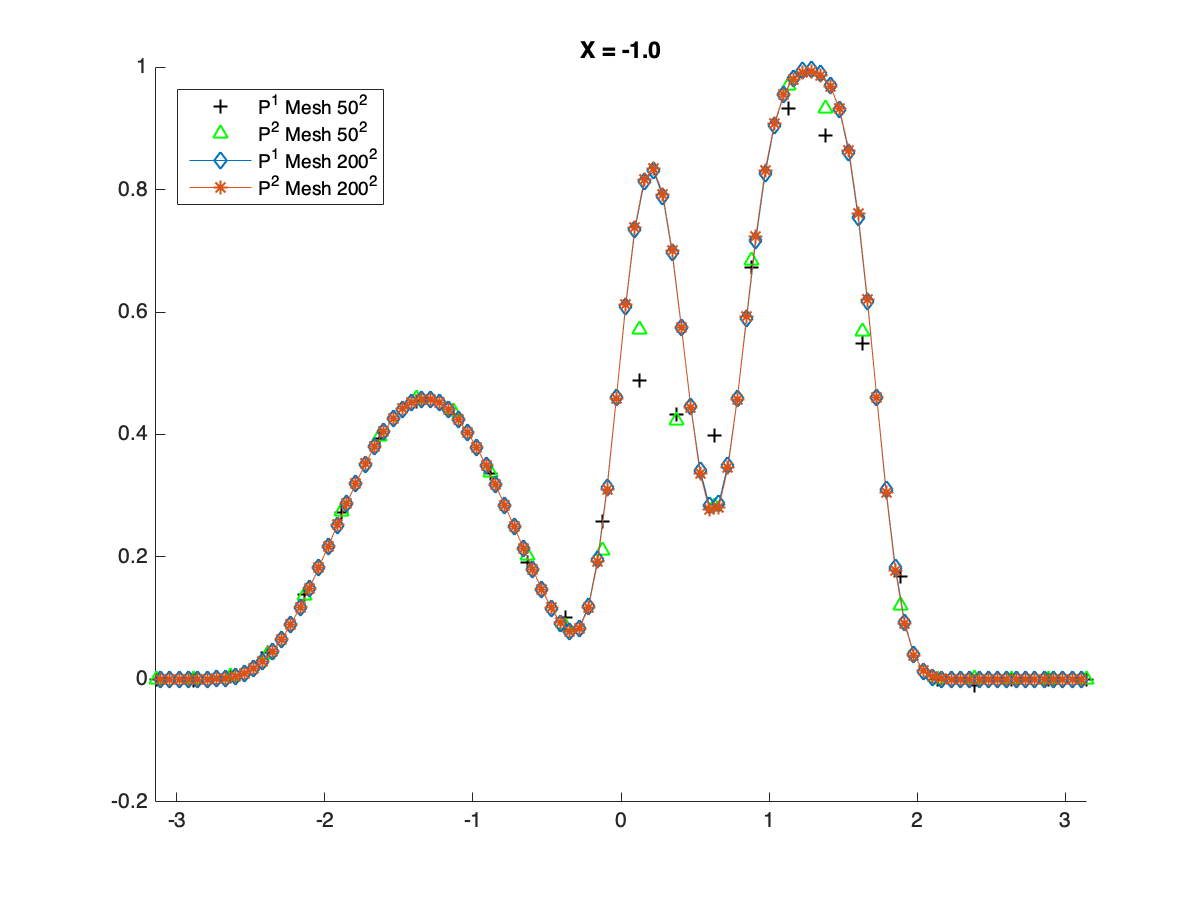}
\end{minipage}
\begin{minipage}[b]{0.45\textwidth}
\centering
    \includegraphics[width=\textwidth]{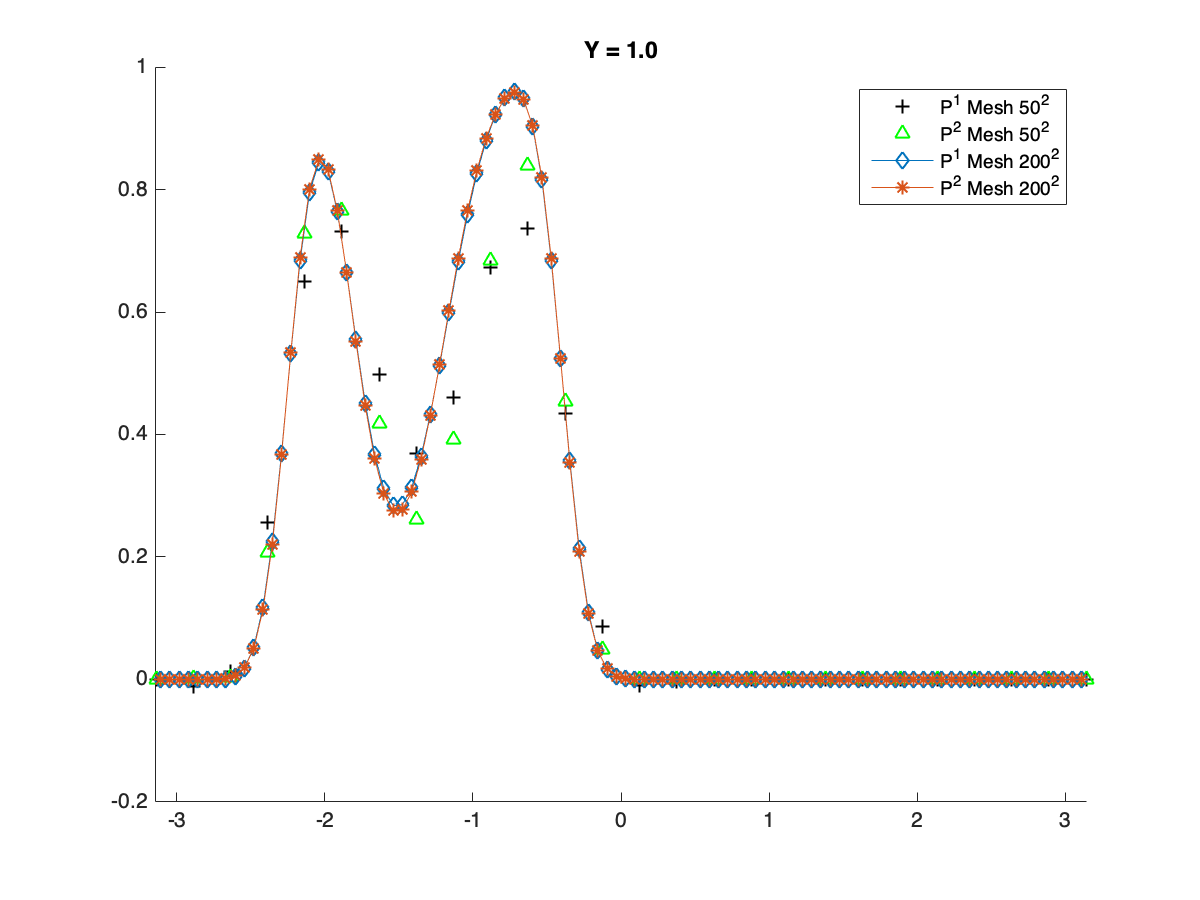}
\end{minipage}
\caption{\small{Plots of the 1D cuts of the numerical solutions for equation~\eqref{eq_exa4.1} at $X = -1.0$ (left) and $Y = 1.0$ (right) with initial data in Figure~\ref{fig:disk_cone_hump} with $CFL = 2.5$ at $T = 1.0$.}}
\label{fig:1D_cuts_rigid}
\end{figure}

%\end{exa}

%---------------------------------------------------------------------------------------------------------

\begin{exa} \label{exa_2d_swirling}

(Swirling deformation flow.) Consider
 \beq \label{eq_exa5}
u_t-( \cos( \frac{x}{2} )^2 \sin(y) f(t) u)_x+( \sin(x) \cos(\frac{y}{2})^2 f(t) u )_y		=	\epsilon \Delta u,	 \quad x, y \in[-\pi, \pi]
\eeq
with $f(t) = \cos( \frac{\pi t}{T} ) \pi$. 

To test the spatial order of accuracy, the initial condition is set to be 
\beq
u(x,y,0) =
\begin{cases}
r^b_0 \cos( \frac{r^b(\mathbf{x})\pi}{2r^b_0} )^6,		\quad & \text{if}\ r^b(\mathbf{x}) < r^b_0,\\
0,										\quad &\text{otherwise},
\end{cases}
\eeq
where $r^b_0 = 0.3\pi$ and $r^b( \mathbf{x} )=\sqrt{ (x-x^b_0)^2+(y-y^b_0)^2 }$ denotes the distance between $(x,y)$ and the center of the cosine bell $(x^b_0,y^b_0) = (0.3 \pi,0)$.
Since there is no  analytical solution available, we choose the numerical solution computed on a refined mesh with $J = 300^2$ cells as a reference solution. The results are  shown in Table \ref{spatial_2d_swir}. Note that comparable performance are observed when quadrilateral and quadratic curved (QC) approximations to the upstream cells are used.  Figure \ref{fig:2D_swirling_temp} shows the temporal convergence study for the proposed scheme coupled with different DIRK methods. A mesh of $J = 200^2$ cells is used. The solution computed with CFL = 0.01 is chosen as the reference solution. The final time is  $0.1$. Again, expected high order temporal convergence rates are observed for all three cases.

Additionally, we solve eq.~\eqref{eq_exa5} with the same initial condition as in Figure \ref{fig:disk_cone_hump}. We choose $\epsilon = 0.01, T = 1.5$ and numerically integrate the solution up to time $1.5$. Numerical solutions and the corresponding contour plots for the proposed schemes are plotted in Figure \ref{fig:disk_cone_hump2}.
The 1D cuts of the numerical solutions of eq.~\eqref{eq_exa5} at $X = 0$ and $Y = 1.54$ with mesh size $J = 50^2$ along with the one computed on a refined mesh with $J = 200^2$ are plotted in Figure \ref{fig:1D_cuts_swirling}. From the 1D cut plot at $X = 0$, both $P^2$ SLDG-LDG and $P^2$ SLDG-LDG-QC schemes perform better than $P^1$ scheme on mesh of size $50^2$.

%----------------------------------
% Spatial order of accuracy (table)

\linespread{1.0}

\begin{table} [htbp]\scriptsize
\centering
\caption{Spatial order of accuracy for Example \ref{exa_2d_swirling} with CFL = 1.0, $\epsilon = 1.0$ at $T = 0.1$.}
\label{spatial_2d_swir}
\bigskip

\begin{tabular}{|c | cc |cc | cc |}
\hline
\multicolumn{7}{|c|}{$k = 0$; Quadrilateral}\\
\hline
 mesh &{$L^1$ error} & Order    &{$L^2$ error} & Order &{$L^{\infty}$ error} & Order\\
\hline
$20^2$	&     2.15E-03 &		 &     6.50E-03 & 			&     5.86E-02 &		 \\
$60^2$ 	&     7.56E-04 &     2.05 &     2.42E-03 &     1.94 &     2.15E-02 &     1.96 \\
$100^2$	&     3.60E-04 &     2.20 &     1.18E-03 &     2.13 &     1.13E-02 &     1.92 \\
$140^2$ &     2.83E-04 &     0.96 &     9.35E-04 &     0.93 &     1.20E-02 &    -0.23 \\
$180^2$ &     2.29E-04 &     1.05 &     7.58E-04 &     1.04 &     8.08E-03 &     1.95 \\
\hline
\multicolumn{7}{|c|}{$k = 1$; Quadrilateral}\\
\hline
 mesh &{$L^1$ error} & Order    &{$L^2$ error} & Order &{$L^{\infty}$ error}& Order\\
\hline
$20^2$ 	&     2.93E-04 & 		&     9.13E-04 & 			&     1.49E-02 &		 \\
$60^2$	&     3.17E-05 &     4.36 &     1.04E-04 &     4.25 &     2.14E-03 &     3.80 \\
$100^2$ &     1.10E-05 &     3.13 &     3.63E-05 &     3.14 &     7.66E-04 &     3.06 \\
$140^2$ &     5.90E-06 &     2.50 &     1.98E-05 &     2.42 &     4.53E-04 &     2.09 \\
$180^2$ &     3.83E-06 &     2.16 &     1.29E-05 &     2.14 &     3.04E-04 &     1.98 \\
\hline
\multicolumn{7}{|c|}{$k = 2$; Quadrilateral}\\
\hline
 mesh &{$L^1$ error} & Order    &{$L^2$ error} & Order &{$L^{\infty}$ error}& Order\\
\hline
$20^2$ 	&     3.42E-05 &		 &     1.10E-04 &			 &     2.73E-03 &		 \\
$60^2$ 	&     1.35E-06 &     6.33 &     4.74E-06 &     6.15 &     1.22E-04 &     6.09 \\
$100^2$ &     3.00E-07 &     4.47 &     1.06E-06 &     4.44 &     2.67E-05 &     4.51 \\
$140^2$	&     1.10E-07 &     3.98 &     3.96E-07 &     3.94 &     1.00E-05 &     3.90 \\
$180^2$ &     5.32E-08 &     3.63 &     1.91E-07 &     3.63 &     5.04E-06 &     3.43 \\
\hline
\multicolumn{7}{|c|}{$k = 2$; QC}\\
\hline
 mesh &{$L^1$ error} & Order    &{$L^2$ error} & Order &{$L^{\infty}$ error}& Order\\
\hline
   $20^2$ &     3.30E-05 &		 &     1.06E-04 & 		&     2.53E-03 & 		\\
    $60^2$ &     1.14E-06 &     6.58 &     3.75E-06 &     6.54 &     9.73E-05 &     6.38 \\
   $100^2$ &     2.48E-07 &     4.54 &     8.15E-07 &     4.54 &     2.09E-05 &     4.58 \\
   $140^2$ &     9.07E-08 &     4.01 &     3.00E-07 &     3.98 &     7.83E-06 &     3.90 \\
   $180^2$ &     4.32E-08 &     3.70 &     1.42E-07 &     3.71 &     3.86E-06 &     3.53 \\
 \hline
\end{tabular}

\end{table}

%----------------------------------
% Temporal order of accuracy (figure)

\begin{figure}
\centering
	\includegraphics[width = 0.5\textwidth]{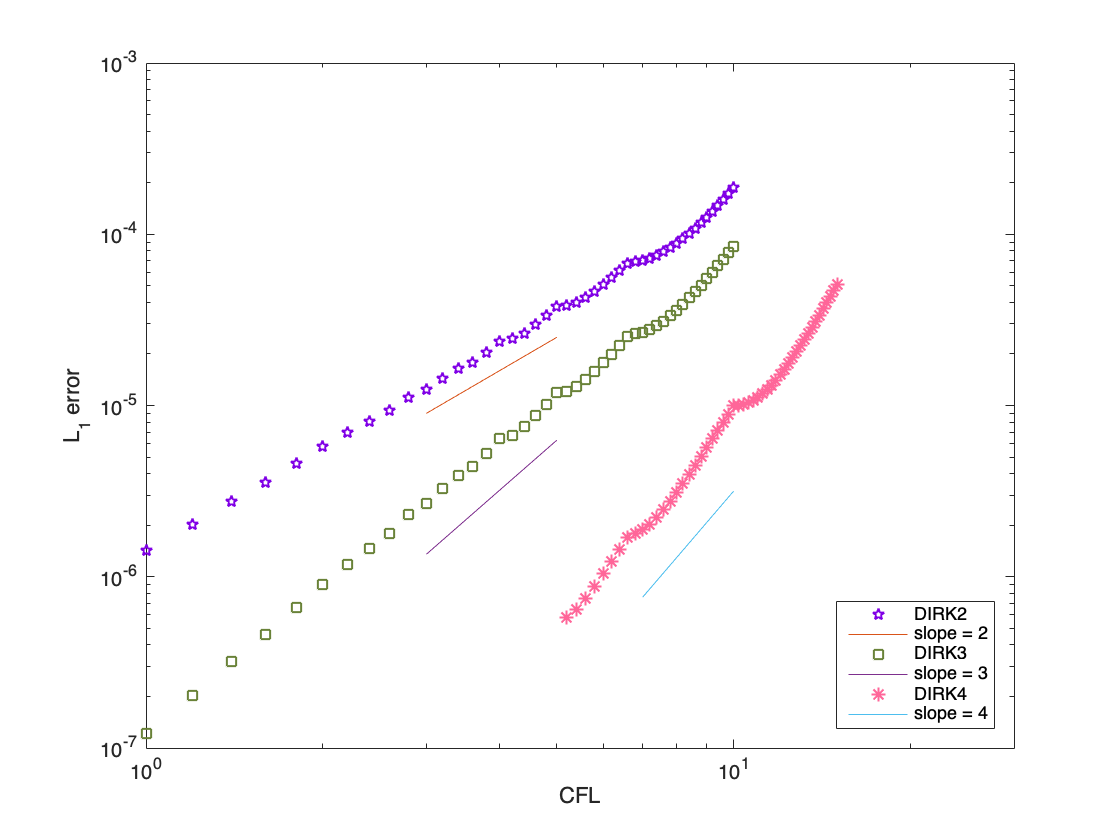}
\caption{$L_1$ error with varying CFL number of Example \ref{exa_2d_swirling}. From top to bottom: DIRK2, DIRK3 and DRIK4.}
\label{fig:2D_swirling_temp}
\end{figure}

\end{exa}

%---------------------------------------------------------------------------------------------------------

\begin{figure}
\begin{minipage}[b]{0.3\textwidth}
\centering
    \includegraphics[width=\textwidth]{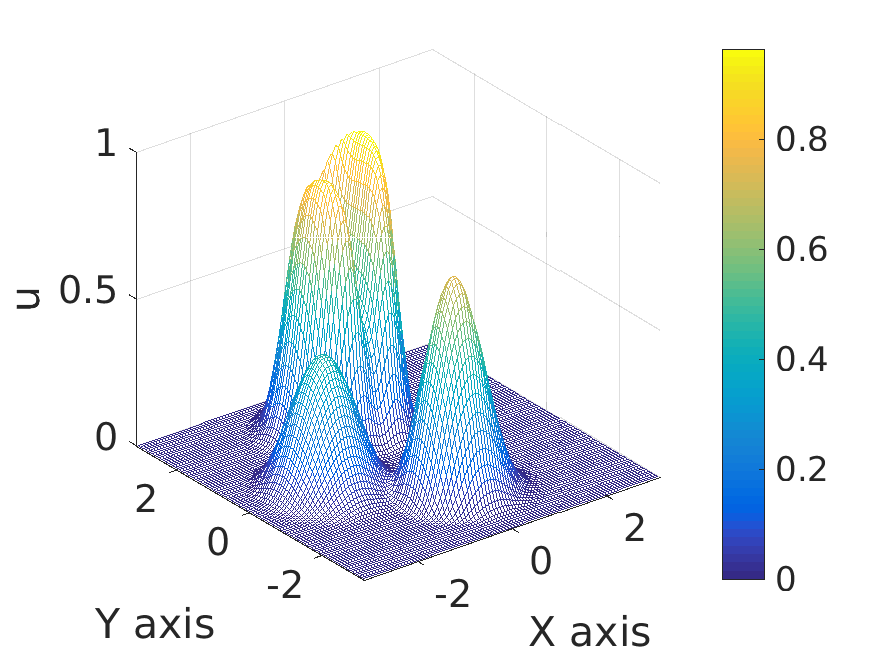}
\end{minipage}
\hfill
\begin{minipage}[b]{0.3\textwidth}
\centering
    \includegraphics[width=\textwidth]{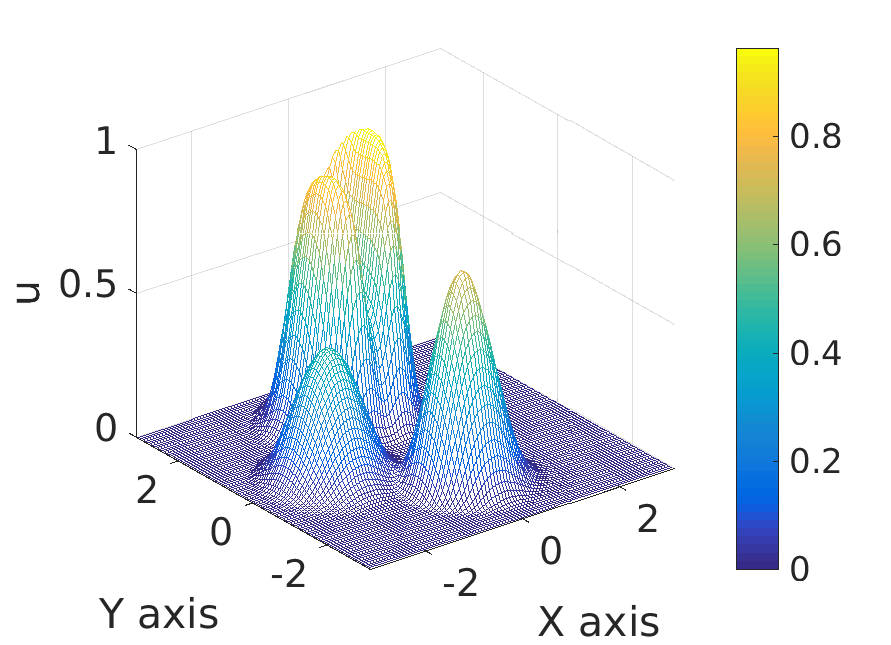}
\end{minipage}
\hfill
\begin{minipage}[b]{0.3\textwidth}
\centering
    \includegraphics[width=\textwidth]{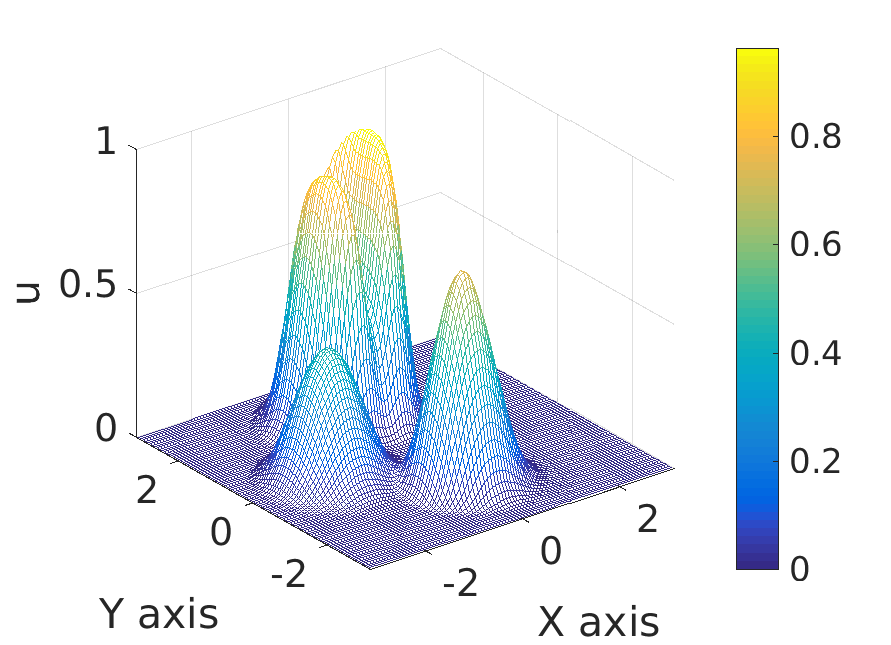}
\end{minipage}

\begin{minipage}[b]{0.3\textwidth}
\centering
    \includegraphics[width=\textwidth]{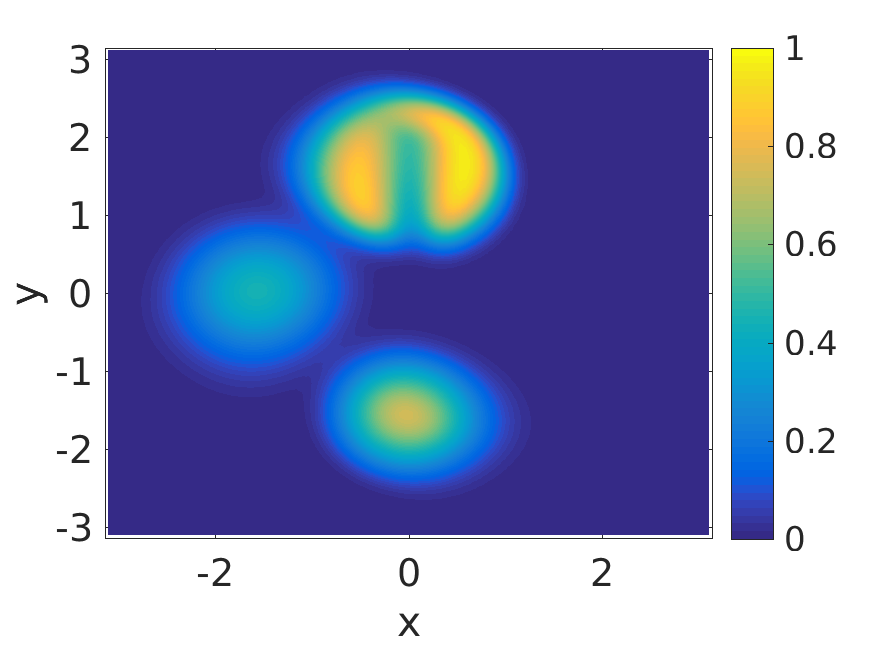}
\end{minipage}
\hfill
\begin{minipage}[b]{0.3\textwidth}
\centering
    \includegraphics[width=\textwidth]{2D_Swirling_plot/nk1_qc0_2Dcontour.png}
\end{minipage}
\hfill
\begin{minipage}[b]{0.3\textwidth}
\centering
    \includegraphics[width=\textwidth]{2D_Swirling_plot/nk1_qc0_2Dcontour.png}
\end{minipage}
\caption{\small{First row consist of mesh plots of SLDG-LDG solutions for equation \eqref{eq_exa5} with $f(t) = \cos( \frac{\pi t}{T} ) \pi $, initial data Figure \ref{fig:disk_cone_hump} and the second row consists of the corresponding contour plots. Mesh size is $100 \times 100$. Final integration time $T = 1.5$. $\Dt = 2.5\Dx$.
From left to right: $P^1$ SLDG-LDG, $P^2$ SLDG-LDG, $P^2$ SLDG-LDG$+$QC. }}
\label{fig:disk_cone_hump2}
\end{figure}

%----------------------------------

\begin{figure}
\begin{minipage}[b]{0.45\textwidth}
\centering
    \includegraphics[width=\textwidth]{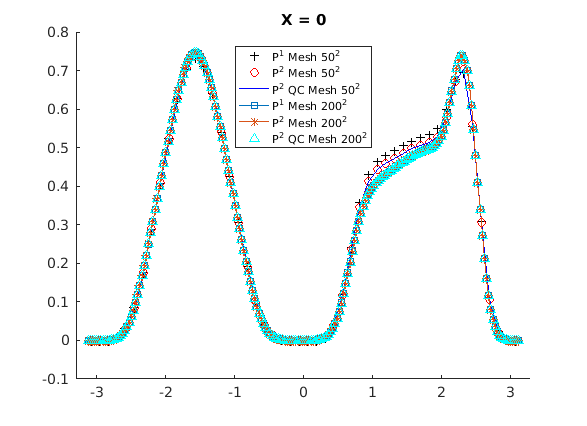}
\end{minipage}
\begin{minipage}[b]{0.45\textwidth}
\centering
    \includegraphics[width=\textwidth]{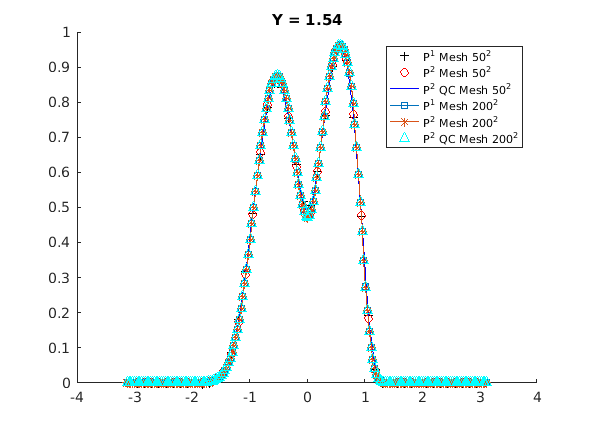}
\end{minipage}
\caption{\small{Plots of the 1D cuts of the numerical solutions for equation \eqref{eq_exa5} at $X = 0$ (left) and $Y = 1.54$ (right) with initial data Figure \ref{fig:disk_cone_hump} at with $CFL = 2.5$ at $T = 1.5$. }}
\label{fig:1D_cuts_swirling}
\end{figure}

\section{
Conclusions
}
\label{sec6}
\setcounter{equation}{0}

In this paper, we developed a semi-Lagrangian (SL) discontinuous Galerkin (DG) method for solving linear convection-diffusion equations. For the scheme formulation,  the DG solution is evolved  along the characteristics to treat the convection part by an efficient SLDG transport method; while the diffusion part is discretized by a local DG method in conjunction with diagonally implicit Runge-Kutta methods along characteristics. The method is high order accurate, mass conservative and is unconditionally stable. In the theoretical aspect, the unconditional $L^2$ stability was proved for the method coupled with the backward Euler discretization. 
In view of the application of the SLDG method coupled with RK exponential integrators to nonlinear Vlasov dynamics in \cite{cai2019high}, extensions of our algorithm to nonlinear convection-diffusion problems will be investigated in our future research work.

\section{Appendix}
\label{sec7}

\subsection{Butcher tableaus for time discretization methods}

In the following Table \ref{tab_dirk3} and Table \ref{tab_dirk4}, we present the third-order DIRK3 and fourth-order DIRK4 respectively. Both tableaus are stiffly accurate.

\linespread{1.5}

% DIRK3

\begin{table}[htbp]
\begin{center}

\begin{tabular}{c|c c c}
	$\gamma$					& 	$\gamma$ 		 	&	  			&					\\
$\f{1+\gamma}{2}$				&	$\f{1-\gamma}{2}$ 		& $\gamma$		&	 	 			\\
		1					&	$\beta_1$ 			& $\beta_2$		&	$\gamma$		\\
\hline
 							& $\beta_1$				& $\beta_2$		& 	$\gamma$		\\
\end{tabular}

\end{center}
\caption{DIRK3. $\gamma \approx 0.435866521508459, \beta_1 = -\f{3}{2} \gamma^2+4\gamma-\f{1}{4}, \beta_2 = \f{3}{2} \gamma^2-5\gamma+\f{5}{4}$.}
\label{tab_dirk3}
\end{table}

% DIRK4

\begin{table}[htbp]
\begin{center}
\begin{tabular}{c|c c c  c c}
	$\f{1}{4}$				&	$\f{1}{4}$ 			& 	 				& 					 & 					& 				\\
	$\f{3}{4}$				&	$\f{1}{2}$ 			&  	$\f{1}{4}$			& 					 & 			      		&				\\
	$\f{11}{20}$			&	$\f{17}{50}$		&	$-\f{1}{25}$ 		&	$\f{1}{4}$		 	 & 					& 				\\
	$\f{1}{2}$			 	&	$\f{371}{1360}$ 	&	$-\f{137}{2720}$	& 	$\f{15}{544}$ 		 &	$\f{1}{4}$			&				 \\
	1					&	$\f{25}{24}$		& 	$-\f{49}{48}$		&	$\f{125}{16}$		 &	$-\f{85}{12}$		& 	$\f{1}{4}$ 		\\
	\hline
						& 	$\f{25}{24}$		&   $-\f{49}{48}$		&	$\f{125}{16}$	 	 & 	$-\f{85}{12}$	         & 	$\f{1}{4}$		\\
\end{tabular}
.
\end{center}
\caption{DIRK4}
\label{tab_dirk4}
\end{table}

%---------------------------------------------------------------------------------------------------------

\subsection{Implementation procedures}
\label{Im_pro}

We now briefly discuss the components of the sparse matrix $B_1$ in \eqref{eq: lsystem}.
Let $\varphi_l\ (l = 0,\cdots,k)$ denote the local bases of $P^k(I_j)$, the numerical solution $u$ can be written as
\begin{equation*}
u = \sum\limits_{l=0}^k u_l \varphi_l \doteq {\bf u} \cdot \boldsymbol{\varphi},\quad x\in I_j.
\end{equation*}
where ${\bf u} = \left(u_0,\cdots,u_k\right)$ and $\boldsymbol{\varphi} = \left(\varphi_0,\cdots,\varphi_k\right)^T$. The basis functions are chosen as the scaled Legendre polynomials. For instance, for $k=2$, $P^2(I_j) = \{1,\xi,\xi^2-\f{1}{12}\}$ with $\xi = \f{x-x_j}{\Dx}$.

To update $u^{n+1}$,
\beq \label{eq:num_test}
(u^{n+1}, \varphi_m)_{I_j}
								=	\left(\sum\limits_{l=0}^k u^{n+1}_l \varphi_l, \varphi_m \right)_{I_j}
								=	\sum\limits_{l=0}^k u^{n+1}_l \left(\varphi_l, \varphi_m \right)_{I_j}.
\eeq
Then, as $\varphi_m$ in \eqref{eq:num_test} going through bases in $P^k(I_j)$,
\begin{equation*}
(u^{n+1}, \boldsymbol{\varphi})_{I_j}
										=	M ({\bf u}^{n+1})^T_{I_j}
\end{equation*}
where $M$ is the mass matrix of size $(k+1)^2$ with $M_{m,l} = \left(\varphi_l, \varphi_m\right)_{I_j} (m,l = 0,1,\cdots,k)$.

Following the procedures in Step 1.2a, the LDG spatial approximation to $\Delta {\bf u}^{n+1}$ can be represented by a matrix vector form as $D_{\Delta} {\bf u}^{n+1}$, where $D_{\Delta}$ is the matrix approximating diffusion operator via the LDG formulation. 

\bit

\item

For \eqref{eq2:12b}, suppose $\hat{u} = u^{-}$, then
$$
\hat{u}^{n+1}_{\jR} \boldsymbol{w}^{-}_{\jR}		=	C ({\bf u}^{n+1})^T_{I_j},	\quad
\hat{u}^{n+1}_{\jL} \boldsymbol{w}^{+}_{\jL}		=	D ({\bf u}^{n+1})^T_{I_{j-1}},	\quad
(u^{n+1},\boldsymbol{w}_x)_{I_j}	=	N ({\bf u}^{n+1})^T_{I_j}
$$

with $(m = 0,1,\cdots,k)$
$$
C_{m,l} =	\varphi_l(\xR) \varphi_m(\xR),	\quad
D_{m,l}	=	\varphi_l(\xL) \varphi_m(\xL),	\quad
N_{m,l} =	(\varphi_l,(\varphi_m)_x)_{I_j}
$$
as $w$ going through $\{\varphi_l\}^k_{l = 0}$.

\item

For \eqref{eq2:12a}, choose $\hat{q} = q^{+}$, then
$$
\hat{q}^{n+1}_{\jR} \boldsymbol{v}^{-}_{\jR}		=	E ({\bf q}^{n+1})^T_{I_{j+1}},	\quad
\hat{q}^{n+1}_{\jL} \boldsymbol{v}^{+}_{\jL}		=	F ({\bf q}^{n+1})^T_{I_j},	\quad
(q^{n+1},\boldsymbol{v}_x)_{I_j}	=	N ({\bf q}^{n+1})^T_{I_j}
$$

with $(m = 0,1,\cdots,k)$
$$
E_{m,l} =	\varphi_l(\xR) \varphi_m(\xL),	\quad
F_{m,l} =	\varphi_l(\xL) \varphi_m(\xL),	
$$
as $v$ going through $\{\varphi_l\}^k_{l = 0}$.

\eit

Then, $(u^{n+1}_{xx},\boldsymbol{\varphi})_{I_j} (\forall j)$ can be represented with a linear combination of the above matrices $C,D,E,F,N$ and ${\bf u}^{n+1}$.

To sum up, $B_1$ in \eqref{eq: lsystem} can be obtained with pre-calculated elements $M,C,D,E,F$ and $N$.

%\eit

%\bibliographystyle{siam}
%\bibliography{References/refer}

\begin{thebibliography}{10}

\bibitem{ascher1997implicit}
{\sc U.~M. Ascher, S.~J. Ruuth, and R.~J. Spiteri}, {\em {Implicit-explicit
  Runge-Kutta methods for time-dependent partial differential equations}},
  Applied Numerical Mathematics, 25 (1997), pp.~151--167.

\bibitem{bonaventura2016flux}
{\sc L.~Bonaventura and R.~Ferretti}, {\em {Flux form Semi-Lagrangian methods
  for parabolic problems}}, Communications in Applied and Industrial
  Mathematics, 7 (2016), pp.~56--73.

\bibitem{bosler2019conservative}
{\sc P.~A. Bosler, A.~M. Bradley, and M.~A. Taylor}, {\em Conservative
  multimoment transport along characteristics for discontinuous galerkin
  methods}, SIAM Journal on Scientific Computing, 41 (2019), pp.~B870--B902.

\bibitem{cai2019high}
{\sc X.~Cai, S.~Boscarino, and J.-M. Qiu}, {\em High order semi-lagrangian
  discontinuous galerkin method coupled with runge-kutta exponential
  integrators for nonlinear vlasov dynamics}, arXiv preprint arXiv:1911.12229,
  (2019).

\bibitem{cai2017high}
{\sc X.~Cai, W.~Guo, and J.-M. Qiu}, {\em {A high order conservative
  semi-Lagrangian discontinuous Galerkin method for two-dimensional transport
  simulations}}, Journal of Scientific Computing, 73 (2017), pp.~514--542.

\bibitem{calvo2001linearly}
{\sc M.~Calvo, J.~De~Frutos, and J.~Novo}, {\em {Linearly implicit Runge--Kutta
  methods for advection--reaction--diffusion equations}}, Applied Numerical
  Mathematics, 37 (2001), pp.~535--549.

\bibitem{celia1990eulerian}
{\sc M.~A. Celia, T.~F. Russell, I.~Herrera, and R.~E. Ewing}, {\em {An
  Eulerian-Lagrangian localized adjoint method for the advection-diffusion
  equation}}, Advances in Water Resources, 13 (1990), pp.~187--206.

\bibitem{cockburn1998local}
{\sc B.~Cockburn and C.-W. Shu}, {\em {The local discontinuous Galerkin method
  for time-dependent convection-diffusion systems}}, SIAM Journal on Numerical
  Analysis, 35 (1998), pp.~2440--2463.

\bibitem{diamantakis2013semi}
{\sc M.~Diamantakis}, {\em The semi-lagrangian technique in atmospheric
  modelling: current status and future challenges}, in ECMWF Seminar in
  numerical methods for atmosphere and ocean modelling, 2013, pp.~183--200.

\bibitem{falcone2013semi}
{\sc M.~Falcone and R.~Ferretti}, {\em Semi-Lagrangian approximation schemes
  for linear and Hamilton-Jacobi equations}, vol.~133, SIAM, 2013.

\bibitem{giraldo2003spectral}
{\sc F.~Giraldo, J.~Perot, and P.~Fischer}, {\em {A spectral element
  semi-Lagrangian (SESL) method for the spherical shallow water equations}},
  Journal of Computational Physics, 190 (2003), pp.~623--650.

\bibitem{groppi2016high}
{\sc M.~Groppi, G.~Russo, and G.~Stracquadanio}, {\em {High order
  semi-Lagrangian methods for the BGK equation}}, Communications
  in Mathematical Sciences, 14 (2016), pp.~389--414.

\bibitem{guo2014conservative}
{\sc W.~Guo, R.~D. Nair, and J.-M. Qiu}, {\em {A conservative semi-Lagrangian
  discontinuous Galerkin scheme on the cubed sphere}}, Monthly Weather Review,
  142 (2014), pp.~457--475.

\bibitem{lauritzen2010conservative}
{\sc P.~H. Lauritzen, R.~D. Nair, and P.~A. Ullrich}, {\em {A conservative
  semi-Lagrangian multi-tracer transport scheme (CSLAM) on the cubed-sphere
  grid}}, Journal of Computational Physics, 229 (2010), pp.~1401--1424.

\bibitem{lin1996multidimensional}
{\sc S.-J. Lin and R.~B. Rood}, {\em {Multidimensional flux-form
  semi-Lagrangian transport schemes}}, Monthly Weather Review, 124 (1996),
  pp.~2046--2070.

\bibitem{russell2002overview}
{\sc T.~F. Russell and M.~A. Celia}, {\em {An overview of research on
  Eulerian--Lagrangian localized adjoint methods (ELLAM)}}, Advances in Water
  resources, 25 (2002), pp.~1215--1231.

\bibitem{sonnendrucker1999semi}
{\sc E.~Sonnendr{\"u}cker, J.~Roche, P.~Bertrand, and A.~Ghizzo}, {\em {The
  semi-Lagrangian method for the numerical resolution of the Vlasov equation}},
  Journal of computational physics, 149 (1999), pp.~201--220.

\bibitem{staniforth1991semi}
{\sc A.~Staniforth and J.~C{\^o}t{\'e}}, {\em {Semi-Lagrangian integration
  schemes for atmospheric models-A review}}, Monthly weather review, 119
  (1991), pp.~2206--2223.

\bibitem{wang2015stability}
{\sc H.~Wang, C.-W. Shu, and Q.~Zhang}, {\em {Stability and error estimates of
  local discontinuous Galerkin methods with implicit-explicit time-marching for
  advection-diffusion problems}}, SIAM Journal on Numerical Analysis, 53
  (2015), pp.~206--227.

\bibitem{wanner1991solving}
{\sc G.~Wanner and E.~Hairer}, {\em {Solving ordinary differential equations
  II}}, Stiff and Differential-Algebraic Problems,  (1991).

\bibitem{xiu2001semi}
{\sc D.~Xiu and G.~E. Karniadakis}, {\em {A semi-Lagrangian high-order method
  for Navier--Stokes equations}}, Journal of computational physics, 172 (2001),
  pp.~658--684.

\end{thebibliography}

\end{document}